\documentclass
[
12pt]
{amsart}%

\usepackage{hyperref}
\usepackage{cite}
\usepackage{amssymb}
\usepackage{amsmath}
\usepackage{amsthm}
\usepackage{amsfonts}
\usepackage{graphicx}
\usepackage{bbm}
\usepackage{xcolor}
\usepackage[toc,page]{appendix}
\usepackage{subfigure}
\usepackage{mathtools}
\usepackage{bm} 
\usepackage[normalem]{ulem}
\usepackage{comment}

\newtheorem{theorem}{Theorem}[section]

\newtheorem{proposition}[theorem]{Proposition}
\newtheorem{corollary}[theorem]{Corollary}
\newtheorem{lemma}[theorem]{Lemma}
\newtheorem{remark}[theorem]{Remark}
\newtheorem{example}[theorem]{Example}

\newtheorem*{acknowledgement}{Acknowledgement}
\theoremstyle{definition}
\newtheorem{definition}[theorem]{Definition}
\newtheorem{notation}[theorem]{Notation}

\numberwithin{equation}{section}

\newcommand{\masha}[1]
{{\color{red} Masha says: #1}}
\newcommand{\marco}[1]
{{\color{blue} Marco says: #1}}

\newcommand{\R}{\mathbb{R}}  
\newcommand{\E}{\mathbb{E}} 
\newcommand{\Prob}{\mathbb{P}}

\newcommand{\Hei}{\mathbb{H}}

\title[sub-Laplacians, spectral properties]{Dirichlet sub-Laplacians on homogeneous Carnot groups: spectral properties, asymptotics, and heat content}

\author[Carfagnini]{Marco Carfagnini{$^{\dag}$}}
\address{ Department of Mathematics\\
University of California, San Diego\\
La Jolla, CA 92093-0112,  U.S.A.}
\email{mcarfagnini@ucsd.edu}

\author[Gordina]{Maria Gordina{$^{\dag }$}}
\thanks{\footnotemark {$\dag$} Research was supported in part by NSF Grant DMS-1954264.}
\address{ Department of Mathematics\\
University of Connecticut\\
Storrs, CT 06269,  U.S.A.}
\email{maria.gordina@uconn.edu}



\begin{document}

\begin{abstract}
We consider sub-Laplacians in open bounded sets in a homogeneous Carnot group and study their spectral properties. We prove that these operators have a pure point spectrum, and prove the existence of the spectral gap. In addition, we give applications to the small
ball problem for a hypoelliptic Brownian motion and the large time
behavior of the heat content in a regular domain.
\end{abstract}

\maketitle

\tableofcontents

\section{Introduction}\label{s.intro}
This paper is focused on spectral properties of sub-Laplacians in subsets of homogeneous Carnot groups. The main difficulties include hypoellipticity of these operators and lack of smoothness of natural sets in such groups such as metric balls. Thus we can not rely on standard partial differential equations' techniques. In particular, we are interested in domains which are balls with respect to a homogeneous distance on such groups. These are classical questions at the intersection of potential theory, spectral analysis and probability in the setting of degenerate second order differential operators.

The mathematical literature on the subject is vast, and while our goal is to prove results of a classical flavor, it seems that they are not readily available. We rely on the theory of Dirichlet forms to show that a sub-Laplacian with Dirichlet boundary conditions on sets with no restrictions on the boundary are infinitesimal generators of a Hunt process, namely, a killed process. To further study boundary behavior of eigenfunctions we address the issue of regularity of boundary points being defined differently in analysis, potential theory and probability.

Our main results include Theorem \ref{thm.spectral.results} where we collect spectral properties of the Dirichlet sub-Laplacian $-\mathcal{L}_{\Omega}$ restricted to a bounded open connected subset $\Omega$ of a Carnot group $\mathbb{G}$. We prove that the spectrum of $-\mathcal{L}_{\Omega}$ is discrete, and then we show that the first Dirichlet eigenvalue is strictly positive and simple, i.e. $-\mathcal{L}_{\Omega}$ has a spectral gap. Moreover, in Proposition~\ref{prop.efunctions.properties} we prove uniform $L^{p}$-bounds and smoothness of the eigenfunctions of $-\mathcal{L}_{\Omega}$. These are well-known results for uniformly elliptic operators on domains with a smooth boundary, see \cite[Chapter 6]{EvansPDEBook2nd},  \cite[Corollary 5.1.2]{SoggeBook2017} \cite[Equation (3.3), p.~39]{ZelditchBook2017}. Our results hold without assuming regularity of the boundary $\partial \Omega$, and in particular apply to uniform elliptic operators on domains with a non-smooth boundary. If the boundary is regular then we can show that the eigenfunctions are zero on the boundary.
As standard PDEs' techniques are of limited use for hypoelliptic operators and for domains with a non-smooth boundary, we rely on the Dirichlet form theory, the Krein-Rutman theorem and irreducibility of the corresponding semigroup.

The spectral analysis in Section \ref{sec.spectral.prop} relies on the Dirichlet form theory on $L^{2} \left( \mathbb{G}, dx \right)$, where $\mathbb{G}$ is a homogeneous Carnot group and $dx$ is a (bi-invariant) Haar measure. A natural question is if these techniques are applicable to more general sub-Riemannian manifolds, but in such a setting we might not have a canonical choice of a measure $m$ which is needed to define a Dirichlet form on the corresponding $L^{2}$ space. We also indirectly rely on the fact that the Haar measure on $\mathbb{G}$ satisfies a volume doubling property. Finally, our approach uses the fact that sub-Laplacians on $L^{2} \left( \mathbb{G}, dx \right)$ are essentially self-adjoint on $C_{c}^{\infty}\left( \mathbb{G} \right)$ in $L^{2} \left( \mathbb{G}, dx \right)$ by \cite[Section 3]{DriverGrossSaloff-Coste2009a}.

Note that \cite{RuzhanskySuragan2017} proves a number of related results for sub-Laplacians on domains in homogeneous Carnot groups. The domains considered there are bounded open with a piecewise smooth and simple boundary. Lower bounds on the spectral gap for Dirichlet sub-Laplacians on compact domains with smooth boundary in sub-Riemannian manifolds have been studied recently in \cite{PrandiRizziSeri2019}. Their methods are different, and in particular, the Dirichlet form theory allows us to consider domains with a non-smooth boundary. Small time asymptotic expansions for hypoelliptic heat kernels  can be found in \cite{ColindeVerdiereHillairetTrelat2021}. As one of the applications is to the heat content asymptotics, we mention that one of the technical difficulties in the sub-Riemannian setting is potential presence of characteristic points on $\partial \Omega$. These are  points $x\in \partial \Omega$ where the horizontal distribution is tangent to $\partial \Omega$. We refer to \cite{TysonWangJ2018, RizziRossi2021} for a more detailed analysis of such points.

We remark that Dirichlet forms in the context of free nilpotent groups have been used in \cite{FrizVictoir2008, FrizVictoirBook2010} in connection with the theory of rough paths. Moreover, they relied on the general Dirichlet forms such as \cite{Sturm1995a, Sturm1996a, Sturm1995b} in the context of free nilpotent groups. In this setting they derived small ball estimates in the context of support theorem for Markovian rough paths.

The paper is organized as follows. In Section \ref{s.Preliminaries} we describe homogeneous Carnot groups, sub-Laplacians, and Dirichlet forms on such groups. In Section \ref{sec.spectral.prop} we describe spectral properties of Dirichlet sub-Laplacians which are collected in Theorem \ref{thm.spectral.results}. In Section \ref{sec.regularpts} we prove that analytic and probabilistic notions of boundary regular points are equivalent. We conclude Section \ref{sec.applications} with two applications: the small ball problem for a hypoelliptic Brownian motion and the large time behavior of the heat content in a regular domain $\Omega$. The latter is done with a natural assumption on the boundary $\partial \Omega$ being regular, therefore  we allow $\partial \Omega$ to contain characteristic points.

\section{Preliminaries}\label{s.Preliminaries}
\subsection{Carnot groups}
In this paper we concentrate on a particular class of nilpotent groups, namely, Carnot groups. We begin by recalling  basic facts about Carnot (stratified) groups. A more detailed description of these spaces can be found in a number of references, see for example \cite{BonfiglioliLanconelliUguzzoniBook, VaropoulosBook1992}.

\begin{definition}[Carnot groups]
We say that $\mathbb{G}$ is a Carnot group of step $r$ if $\mathbb{G}$  is a connected and simply connected Lie group whose Lie algebra $\mathfrak{g}$ is \emph{stratified}, that is, it can be written as
\begin{equation}\label{eqn.strat.lie.alg}
\mathfrak{g}=V_{1}\oplus\cdots\oplus V_{r},
\end{equation}
where
\begin{align*}
& \left[V_{1}, V_{i-1}\right]=V_{i}, \hskip0.1in 2 \leqslant i \leqslant r,
\\
& [ V_{1}, V_{r} ]=\left\{ 0 \right\}.
\end{align*}
\end{definition}
The stratification \eqref{eqn.strat.lie.alg} is not unique as pointed out in  \cite[Section 2.2.1]{BonfiglioliLanconelliUguzzoniBook}. Moreover, if $(V_{1}, \ldots , V_{r} )$ and $(\tilde{V}_{1}, \ldots, \tilde{V}_{s} )$ are two stratification of a  Carnot group $\mathbb{G}$ then $r=s$, which is referred as the step, and  $\dim V_{i} = \dim \tilde{V}_{i}$ for every $i=1, \ldots ,r$ \cite[Proposition 2.28]{BonfiglioliLanconelliUguzzoniBook}. If $\mathbb{G}$ is a Carnot group with stratification $(V_{1}, \ldots , V_{r} )$, by \cite[Proposition 2.2.8]{BonfiglioliLanconelliUguzzoniBook} it follows that the step $r$ and the number $m:=\dim V_{1}$ of generators of $\mathbb{G}$ do not depend on the chosen stratification.

For the rest of the paper we fix a stratification $(V_{1}, \ldots, V_{s} )$ of $\mathbb{G}$. This stratification determines the space $\mathcal{H} := V_{1}$ of horizontal vectors that generate the rest of the Lie algebra, nothing that $V_2=[\mathcal{H}, \mathcal{H}], ..., V_r = \mathcal{H}^{(r)}$. The horizontal space $\mathcal{H}$ is used to construct a sub-Laplacian on $\mathbb{G}$ for which we refer to \cite[Section 2.2]{BonfiglioliLanconelliUguzzoniBook}.

To avoid degenerate cases we assume that the dimension of the Lie algebra $\mathfrak{g}$ is at least $3$ which implies $\dim V_{1} \geqslant 2$. 
We generally assume that $r \geqslant 2$ to exclude the case when the corresponding Laplacian is elliptic.

In particular, Carnot groups are nilpotent. We will use $\mathcal{H}:=V_{1}$ to denote the space of \textit{horizontal} vectors that generate the rest of the Lie algebra, noting that $V_2=[\mathcal{H}, \mathcal{H}], ..., V_r = \mathcal{H}^{(r)}$. As usual, we let
\begin{align*}
\exp&: \mathfrak{g} \longrightarrow \mathbb{G},
\\
\log&:\mathbb{G} \longrightarrow \mathfrak{g}
\end{align*}
denote the exponential and logarithmic maps, which  are global diffeomorphisms for connected nilpotent groups, see for example \cite[Theorem 1.2.1]{CorwinGreenleafBook}.

Finally, by \cite[Proposition 2.2.17, Proposition 2.2.18]{BonfiglioliLanconelliUguzzoniBook} we can assume without loss of generality that a Carnot group can be identified with a \emph{homogeneous Carnot group}. For $i=1, ..., r$, let $d_{i}=\operatorname{dim} V_{i}$ and $d_0=0$. The Euclidean space underlying $\mathbb{G}$ has the dimension
\begin{align*}
N:=\sum_{i=1}^{r} d_{i},
\end{align*}
and the homogeneous dimension of $\mathbb{G}$ is given by
\begin{equation}\label{e.HomogeneousDim}
Q:=\sum_{i=1}^{r} i \cdot  d_{i}.
\end{equation}
The identification of $\mathbb{G}$ with $\mathbb{R}^{N}$ allows us to define exponential coordinates as follows.

\begin{definition}\label{d.2.1} We say that a collection of left-invariant vector fields $\left\{ X_{1}, ..., X_{N} \right\}$  on a Carnot group $\mathbb{G}$ is a \emph{basis for $\mathfrak{g}$ adapted to the stratification} if the set $\left\{ X_{d_{i-1}}, ..., X_{ d_{i-1}+d_{i}} \right\}$ is a basis of $V_{i}$ for each $i=1,\ldots,r$.
\end{definition}
The fact that $\exp$ is a global diffeomorphism can be used to parameterize $\mathbb{G}$ by its Lie algebra $\mathfrak{g}$. First we recall the Baker-Campbell-Dynkin-Hausdorff formula.
\begin{notation}\label{n.BCDH}
For any $X, Y \in \mathfrak{g}$, we define the \emph{Baker-Campbell-Dynkin-Hausdorff formula}  by
\begin{align*}
& \operatorname{BCDH}\left( X, Y \right):=\log\left( \exp X \exp Y \right)
\\
& =\sum\limits_{m=1}^{\infty} \sum\limits_{p_{i}+q_{i}\geqslant 1 }
\frac{(-1)^{m-1}}{m(\sum\limits_{i=1}^{m} p_i+q_i)}
\\
&
\times\frac{
	\overbrace{[[\cdots[X, X],\cdots], X}^{p_1}
],
\overbrace{Y], \cdots, Y}^{q_1}
],\cdots
\overbrace{\cdots, X], \cdots X}^{p_m}
	\overbrace{\cdots,Y],\cdots Y}^{q_m}
]}{p_1!q_1!\cdots p_m!q_m!}.
\end{align*}
\end{notation}

\begin{definition}\label{d.ExpCoord} Suppose $\left\{ X_{1}, \dots, X_{N} \right\}$ is a basis of $\mathfrak{g}$ adapted to the stratification. For a point $g \in \mathbb{G} \cong \mathbb{R}^{N}$ we say that $\left( x_{1}, \dots, x_{N} \right) \in \mathbb{R}^{N}$ are \emph{exponential coordinates of the first kind} relative to the basis $\left\{ X_{1}, \dots, X_{N} \right\}$  if
\[
g= \exp\left( \sum_{i=1}^{N} x_{i}X_{i} \right).
\]
We equip $\mathbb{R}^{N}$ with the group operation pulled back from $\mathbb{G}$ by
\begin{align*}
z&:=x \star y,
\\
\sum_{i=1}^{N} z_{i}X_{i}&=\operatorname{BCDH}\left( \sum_{i=1}^{N} x_{i}X_{i}, \sum_{i=1}^{N} y_{j}X_{j} \right).
\end{align*}
\end{definition}
In particular, in this identification $x^{-1}=-x$. Note that $\mathbb{R}^{N}$ with this group law is a Lie group whose Lie algebra is isomorphic to $\mathfrak{g}$. Both $\mathbb{G}$ and $\left( \mathbb{R}^{N}, \star \right)$ are nilpotent connected and simply connected, therefore the exponential coordinates give a diffeomorphism between $\mathbb{G}$ and $\mathbb{R}^{N}$.

A stratified Lie algebra is equipped with a natural family of \emph{dilations} defined for any $a>0$ by
\[
\delta_{a}\left( X \right):=a^i X, \qquad \text{for } X \in V_{i}.
\]
For each $a>0$, $\delta_a$ is a Lie algebra isomorphism, and the family of all dilations $\{ \delta_a \}_{a>0}$ forms a one-parameter group of Lie algebra isomorphisms. We again can use the fact that $\exp$ is a global diffeomorphism to define the automorphisms $D_{a}$ on $G$. The maps $D_{a}:=\exp \circ \, \delta_{a} \circ \log: \mathbb{G} \longrightarrow \mathbb{G}$ satisfy the following properties.
\begin{equation} \label{e.Dilations}
\begin{split}
&  D_{a}\circ \exp=\exp \circ\, \delta_{a} \quad\text{ for any } a >0,
\\
&  D_{a_{1}}\circ D_{a_{2}}=D_{a_{1}a_{2}}, D_{1}=I \quad\text{ for any } a_{1}, a_{2}>0,
\\
&  D_{a}\left( g_{1} \right)D_{a}\left( g_{2} \right)=D_{a}\left( g_{1}g_{2} \right) \quad\text{ for any } a>0 \text{ and } g_{1}, g_{2} \in G,
\\
& dD_{a}=\delta_{a}.
\end{split}\end{equation}
That is, the group $\mathbb{G}$ has a family of dilations which is adapted to its stratified structure. Actually $D_{a}$ is the unique Lie group automorphism corresponding to $\delta_{a}$.
On a homogeneous Carnot group $\mathbb{R}^{N}$ the dilation $D_{a}$ can be described explicitly by
\begin{equation*}
D_{a}\left( x_{1}, \dots, x_{N} \right):=\left( a^{\sigma_{1}}x_{1}, \dots, a^{\sigma_{N}}x_{N}\right),
\end{equation*}
where $\sigma_{j} \in \mathbb{N}$ is called the \emph{homogeneity} of $x_{j}$, with
\[
\sigma_{j}= i \quad\text{ for } \sum_{k=0}^{i-1}d_{k} +1 \leqslant j \leqslant \sum_{k=1}^{i}d_{k},
\]
with $i=1,\ldots,r$ and recalling that $d_0=0$.
That is, $\sigma_{1}=\dots= \sigma_{d_{1}}=1, \sigma_{d_{1}+1}=\cdots=\sigma_{d_1+d_2}=2$, and so on.

We assume that $\mathcal{H}$ is equipped with an inner product $\langle \cdot, \cdot\rangle_{\mathcal{H}}$, in which case the Carnot group has a natural sub-Riemannian structure. Namely, one may use left translation to define a \emph{horizontal distribution} $\mathcal{D}$ as a sub-bundle of the tangent bundle $T\mathbb{G}$, and a metric on $\mathcal{D}$. First, we identify the space $\mathcal{H} \subset \mathfrak{g}$ with $\mathcal{D}_{e}\subset T_e\mathbb{G}$. Then for $g\in \mathbb{G}$ let $L_{g}$ denote left translation $L_{g}h =gh$, and define $\mathcal{D}_{g}:=(L_g)_{\ast}\mathcal{D}_{e}$ for any $g \in \mathbb{G}$. A metric on $\mathcal{D}$ may then be defined by
\begin{align*}
\langle u, v \rangle_{\mathcal{D}_g} &:= \langle  (L_{g^{-1}})_{\ast} u, (L_{g^{-1}})_{\ast}v\rangle_{\mathcal{D}_e}
\\
&= \langle (L_{g^{-1}})_{\ast}u,(L_{g^{-1}})_{\ast}v\rangle_{\mathcal{H}} \qquad \text{for all } u, v \in\mathcal{D}_g.
\end{align*}
We will sometimes identify the horizontal distribution $\mathcal{D}$ and $\mathcal{H}$. Vectors in $\mathcal{D}$ are called \emph{horizontal}.

\subsection{Sub-Laplacians and Dirichlet forms}
For sub-Laplacians on sub-Riemannian complete manifolds  \cite[pp. 41-42]{Strichartz1986a} claimed that these operators are essentially self-adjoint, though without a complete proof or indication of  how to choose the reference measure. If the manifold is a Lie group, a natural choice is a Haar measure, and we review relevant results below. For a more recent approach we refer to  \cite{FranceschiPrandiRizzi2020}. To tackle more general sub-Riemannian manifolds in the future one might also use the semigroup approach in \cite{GordinaLaetsch2016a, GordinaLaetsch2017} combined with the Dirichlet form theory in \cite{MaRocknerBook, FukushimaOshimaTakedaBook2011}.

Suppose $G$ is a real connected Lie group, and $\left\{ X_{i} \right\}_{i=1}^{m}$ is a family of left-invariant vector fields on $G$ satisfying H\"{o}rmander's condition, then the  sum of squares operator
\[
\mathcal{L}:=\sum_{i=1}^{m} X_{i}^{2}
\]
is essentially self-adjoint on $C_{c}^{\infty}\left( G \right)$ in $L^{2}\left( G, dx \right)$ according to \cite[p.~950]{DriverGrossSaloff-Coste2009a}. Here $dx$ is a (right) Haar measure on $G$. In particular, the following integration by parts formula holds
\[
\int_{G} \sum_{i=1}^{m}X_{i}^{2}f dx=-\langle \mathcal{L}^{\ast}f, f\rangle_{L^{2}\left( G, dx \right)} < \infty
\]
for any $f \in C^{\infty}\left( G \right) \cap \mathcal{D}_{\mathcal{L}^{\ast}}$.
If $G=\mathbb{G}$ is a homogeneous Carnot group and $\left\{ X_{1}, ..., X_{d_{1}} \right\}$ is an orthonormal basis of $\mathcal{H}$ of left-invariant vector fields and $m=d_{1}$, then by \cite[Section 3]{DriverGrossSaloff-Coste2009a} the operator $\mathcal{L}$ depends only on the inner product on $\mathcal{H}$, and not on the choice of the basis. This is what \cite[Example 1.5.2]{BonfiglioliLanconelliUguzzoniBook} calls a canonical sub-Laplacian. More background on sub-Laplacians can be found in \cite[Section 3.1]{GordinaLaetsch2016a}. This is not the subject of this paper, though we mention that by \cite[Corollary 3.4]{GordinaLaetsch2016a} all sub-Laplacians differ only by first-order terms.  We abuse notation and denote by $X_{j}$ both the vector in $\mathcal{H}$ and the unique extension of this vector to the left-invariant vector field on $\mathbb{G}$.

\begin{remark}[Choice of the measure]\label{r.Measure}We chose a Haar measure as the reference measure for several reasons. First of all, the sub-Laplacians we are interested in are essential self-adjoint on $C_{c}^{\infty}\left( \mathbb{G} \right)$ in $L^{2}\left( \mathbb{G}, dy \right)$ by \cite[Section 3]{DriverGrossSaloff-Coste2009a}. Secondly, a Haar measure on the metric space $\left( \mathbb{G}, \vert \cdot \vert \right)$ is doubling which we need for using heat kernel estimates in \cite{Sturm1996a} The fact that a Haar measure is doubling in this setting follows from \cite[Theorem 3.2]{DriverGrossSaloff-Coste2009a} and \cite{JerisonSanchez-Calle1987, VaropoulosBook1992, NagelSteinWainger1985}. For more comments on how the sub-Laplacians might depend on the choice of a reference measure including on when we can write it as $\operatorname{div} \nabla_{\mathcal{H}}$  we refer to \cite[Section 4]{GordinaLaetsch2016a}.
\end{remark}

Before listing basic properties of this second order differential operator, recall that $\mathbb{G}$ is unimodular, and so we can assume that it is equipped with a (bi-invariant) Haar measure $dx$. Moreover, if $\mu$ is the push-forward of the $N$-dimensional Lebesgue measure $\mathcal{L}^{N}$ via the exponential map, then it is a bi-invariant Haar measure on $\mathbb{G}$ and
\begin{equation}\label{e.HaarScaling}
d\mu\left( y \circ \delta_{\lambda} \right) = \lambda^{Q}d\mu\left( y \right).
\end{equation}
If we identify $\mathbb{G}$ with the homogeneous Carnot group $\left( \R^N, \star, \delta_\lambda \right)$ via exponential coordinates in Definition \ref{d.ExpCoord}, then for a measurable set $E \subset \mathbb{G}$, its Haar measure is explicitly given by $\mu\left( E \right) =\mathcal{L}^{N}\left( \exp^{-1}\left(E\right) \right)$. We will abuse notation and use the same notation for both measures.

As observed in \cite[Section 3]{DriverGrossSaloff-Coste2009a}, the operator $\mathcal{L}$ is essentially self-adjoint on $C_{c}^{\infty}\left( \mathbb{G} \right)$ in $L^{2}\left( \mathbb{G}, dy \right)$. Thus the corresponding semigroup $e^{t \mathcal{L}}$ can be defined by the spectral theorem. This semigroup commutes with left translations since its infinitesimal generator is left-invariant as well. Namely, for every $y\in \mathbb{G}$ we have
\[
\mathcal{L} (f (y \circ x)) = ( \mathcal{L} f) (y \circ x) \quad \text{ for every } x \in \mathbb{G} \;\;\text{ and } f \in C_{c}^\infty (\mathbb{G}).
\]

We now recall some basic properties of such sub-Laplacians, for more properties including regularity of the corresponding heat kernel, parabolic Harnack inequality, Gaussian upper and lower bounds we refer to \cite[Theorem 3.4]{DriverGrossSaloff-Coste2009a} in a more general setting. Some of these properties rely on the fact that $\mathcal{L}$ is hypoelliptic by H\"ormander's hypoellipticity theorem \cite[Theorem 1.1]{Hormander1967a}, and doubling property for the metric we describe below.
In addition, on Carnot groups the operator $\mathcal{L}$ is $\delta_\lambda$-homogeneous of degree two, that is, for every fixed $\lambda >0$
\[
\mathcal{L} (f (\delta_\lambda ( x) ))  = \lambda^2 ( \mathcal{L} f) (  \delta_\lambda ( x)) \quad \text{ for every } \;\; x \in \mathbb{G} \;\;\text{ and } f \in C_{c}^\infty (\mathbb{G}),
\]
since vector fields $X_j$s are $\delta_\lambda$-homogeneous of degree one. For more details we refer to \cite[p.~63]{BonfiglioliLanconelliUguzzoniBook} on homogeneous Carnot groups.

The sub-Laplacian is symmetric since the group is unimodular, and hence the integration by parts formula reads

\begin{align}\label{eqn.int.by.parts1}
\int_{\mathbb{G}} \sum_{j=1}^{d_{1}} X_{j}f \cdot X_{j}g dy=- \int_{\mathbb{G}} \mathcal{L} f\cdot  g dy, f, g \in C_{c}^{\infty}\left( \mathbb{G} \right).
\end{align}

The symmetric form corresponding to the semigroup $e^{t\mathcal{L}}$ is
\begin{align}\label{e.DirichletForm}
& f \rightarrow \mathcal{E}\left( f \right):=\int_{\mathbb{G}} \vert \nabla_\mathcal{L} f (y) \vert^{2}_{\R^{d_{1}}}  dy,
\\
&  \text{ where } \nabla_\mathcal{L}:= \left( X_1, ...,  X_{d_{1}} \right)  \notag
\end{align}
is the horizontal gradient, and
\begin{align*}
& \mathcal{D}_{\mathcal{E}} :=  W^{1}_{2}\left( \mathbb{G}\right)
\\
& :=\left\{ f \in L^{2}(\mathbb{G},dx)  :  X_{i} f \in L^{2}(\mathbb{G}, dx) ,  \text{ for all} \; i=1,\ldots, d_{1} \right\},
\end{align*}
where $X_{i}f$ is to understood in the distributional sense. The form $\mathcal{E}$ is a Dirichlet form by \cite[p. 951]{DriverGrossSaloff-Coste2009a}. Note that $\mathcal{E}$ is a closed form, that is, $\mathcal{D}_{\mathcal{E}}$ is a Hilbert space with respect to the inner product
\begin{align*}
& \langle f,g \rangle_{W^{1}_{2}\left( \mathbb{G}\right)}  = \mathcal{E}(f,g) + \langle f,g \rangle_{L^{2}\left( \mathbb{G},dx\right)},
\end{align*}
where $\mathcal{E}(f,g)$ is obtained from \eqref{e.DirichletForm} by polarization. Moreover, by \eqref{eqn.int.by.parts1} we have that for $f \in \mathcal{D}_{\mathcal{E}}$

\begin{align}\label{eqn.int.by.parts}
& \int_{\mathbb{G}} \vert \nabla_\mathcal{L} f\left( y \right) \vert^2_{\mathcal{H}}dy = - \int_{\mathbb{G}}f  (y) \mathcal{L} f (y) dy.
\end{align}

The next step is to check that \eqref{e.DirichletForm} can be extended to a regular Dirichlet form. Recall that a Dirichlet form $\left( \mathcal{E}, \mathcal{D}_{\mathcal{E}} \right)$ is called \emph{regular} if it admits a core, that is, if there exists a subset $\mathcal{C}$ of $\mathcal{D}_{\mathcal{E}}\cap C_{0}\left( \mathbb{G} \right)$ that is dense both in $\mathcal{D}_{\mathcal{E}}$ with respect to the Sobolev norm $\Vert \cdot \Vert_{W^{1}_{2}\left( \mathbb{G}\right)}$, and in $C_{0}\left( \mathbb{G} \right)$ with respect to the $\sup$-norm. Note that $C_{c}^{\infty}\left( \mathbb{G} \right)$ is dense in $L^{2}\left( G, dx \right)$, and it is a core for the bilinear form $\mathcal{E}$. Thus, $\left( \mathcal{E}, \mathcal{D}_{\mathcal{E}} \right)$ is a regular Dirichlet form. In addition it is strongly local as defined in \cite[Definition 1.2]{GrigoryanHuLau2014}.

\begin{definition}\label{dfn.hom.norm}
Suppose $\mathbb{G} = \left( \R^N, \star, \delta_\lambda \right)$ is a homogeneous Carnot group, and $\rho:\mathbb{G} \rightarrow [0,\infty)$ is a continuous function with respect to the Euclidean topology. Then $\rho$ is a \emph{homogeneous norm} if it satisfies the following properties
\begin{align*}
& \rho\left( \delta_{\lambda} (x) \right) = \lambda \rho(x)  \text{ for every }  \lambda >0   \text{ and } x \in \mathbb{G},
\\
& \rho(x) >0  \text{ if and only if }  x\not=0.
\end{align*}
The norm $\rho$ is called \emph{symmetric} if it satisfies $\rho\left(x^{-1} \right) = \rho\left( x \right)$ for every $x \in \mathbb{G}$.
\end{definition}
If $\rho_{1}$ and $\rho_{2}$ are two homogeneous norms, then there exists a constant $c>0$ such that
\begin{equation}\label{eqn.equiv.hom.norms}
c^{-1} \rho_1(x) \leqslant \rho_2(x) \leqslant c \rho_1(x), \text{ for every } x \in \mathbb{G},
\end{equation}
see e.g. \cite[Proposition 5.1.4 p. 230]{BonfiglioliLanconelliUguzzoniBook}.
On every homogeneous Carnot group there exist  distinguished symmetric homogeneous norms related to a sub-Laplacian as follows.

\begin{definition}\label{dfn.gauges} A homogeneous symmetric norm $\rho$ on $\mathbb{G}$ is called an \emph{$\mathcal{L}$-gauge} if it is smooth everywhere except at the origin and
\begin{equation}\label{eqn.gauge}
\mathcal{L} (\rho^{2-Q} (x)) = 0,   x \in \mathbb{G}\setminus \left\{ 0 \right\},
\end{equation}
where $Q$ is the homogeneous dimension of $\mathbb{G}$.
\end{definition}
By \cite[Section 5.3]{BonfiglioliLanconelliUguzzoniBook} we know that  there exists a unique fundamental solution $\Gamma$ for the Poisson equation with a sub-Laplacian $\mathcal{L}$, that is, $\Gamma \in C^{\infty} \left( \mathbb{G} \setminus \{ 0\} \right) \bigcap L^1_{\operatorname{loc}} (\mathbb{G}, dx)$  and
\[
\mathcal{L} \Gamma = - \operatorname{Dirac}_0,
\]
where $\operatorname{Dirac}_{0}$ is the Dirac measure supported at $\{ 0 \}$.

$\mathcal{L}$-gauges and the fundamental solution for $\mathcal{L}$ are related as follows. If $\mathbb{G}$ is a Carnot group of homogeneous dimension $Q$ and $\Gamma$ is a fundamental solution for $\mathcal{L}$, then
\begin{equation}
\rho(x):=
\left\{
\begin{array}{ll}
\Gamma(x)^{\frac{1}{2-Q}}, & x\in \mathbb{G}\backslash \{ 0 \};
\\
0, & x=0.
  \end{array}
\right.
\end{equation}
is an $\mathcal{L}$-gauge on $\mathbb{G}$. By \cite[Section 5.5]{BonfiglioliLanconelliUguzzoniBook} if $\rho$ is an $\mathcal{L}$-gauge on $\mathbb{G}$, then there exists a constant $\alpha_d$ such that  $\Gamma = \alpha_{\rho} \rho^{2-Q}$ is Green's function for $\mathcal{L}$. As a consequence, the $\mathcal{L}$-gauge is unique up to a multiplicative constant.

\subsection{Regular boundary points: a probabilistic approach}

We now recall the notion of regular points for a differential operator $\mathcal{L}$ in a bounded open connected set $\Omega \subset \mathbb{R}^{N}$ as found in a number of references, in particular for Brownian motion in \cite[Chapter 8]{SchillingPartzschBook2012}. The connection with classical potential theory goes back to J.~Doob,  E.~B.~Dynkin,  M.~Kac \cite{Dynkin1959a, Doob1954, Kac1951} et al.

Suppose that $g_{\, \cdot} := \left\{ g_{t} \right\}_{t}$ is a $\mathbb{G}$-valued diffusion process whose infinitesimal generator is $\mathcal{L}$. Using exponential coordinates of the first type, we can view $g_{\, \cdot}$ as an $\R^{N}$-valued process. We start with a probabilistic definition of regular points as found in \cite[Definition 8.2.1]{KallianpurSundar2014}, and this is the definition used in \cite{Gaveau1977a} for $\mathcal{L}=-\frac{1}{2} \Delta_\mathcal{H}$ on $\mathbb{R}^{3} \cong \Hei$, the three-dimensional Heisenberg group.

\begin{definition}\label{d.regular.set.prob} Let $g_{\, \cdot}^{x}$ be the diffusion process  with generator $\mathcal{L}$ started at $x \in \partial \Omega$, where $\Omega$ is a bounded open connected set, and define
\[
\tau^{x}_{\Omega} := \inf \left\{ t > 0: g^{x}_{t} \in \Omega^{c} \right\}.
\]
We call $x$ a \emph{regular point} of $\partial \Omega$ if $\mathbb{P}^{x}\left( \tau^{x}_{\Omega} =0 \right)=1$. We call the set $\Omega$ \emph{regular} if every boundary point of $\Omega$ is regular. If $\mathbb{P}^{x}\left( \tau^{x}_{\Omega}=0 \right)<1$,  the point $x$ is called a \emph{singular point} of the boundary.
\end{definition}

In Section \ref{sec.regularpts} we compare Definition \ref{d.regular.set.prob} with an analytic definition of regular points used for hypoelliptic operators such as sub-Laplacians on homogeneous Carnot groups. Moreover, we prove that the two notion of regular points are equivalent.

\section{Spectral properties of the sub-Laplacian $\mathcal{L}$ restricted to a set $\Omega$}\label{sec.spectral.prop}

We rely on the Dirichlet form theory to describe a restriction of $\mathcal{L}$ to a set $\Omega$ in $\mathbb{G}$. Our main reference here is \cite{FukushimaOshimaTakedaBook2011}, and in a more relevant setting of metric measure spaces \cite[Section 6.1]{Grigor'yan2010a} and \cite[p.~173]{GrigoryanHuLau2014}. In particular, we use this approach to show that $\mathcal{L}_{\Omega}$ has a discrete spectrum with minimal assumptions on the boundary of $\Omega$.

Let $\Omega$ be a bounded open connected set in $\mathbb{G}$ and define
\[
\mathcal{D}_{\mathcal{E}} (\Omega) :=  \overline{\left\{ f\in \mathcal{D}_{\mathcal{E}}:  \operatorname{supp} f \subset \Omega \right\}}^{W_2^1 },
\]
where $\Vert f \Vert^{2}_{W_2^1 } = \mathcal{E}(f) + \Vert f \Vert^{2}_{L^{2} (\mathbb{G},dx)}$, and $(\mathcal{E}, \mathcal{D}_{\mathcal{E}})$ is given in \eqref{e.DirichletForm}. Then by  \cite[Section 3.2 Theorem 3.3]{GrigoryanHuLau2014} we know that $\left( \mathcal{E}, \mathcal{D}_{\mathcal{E}} (\Omega) \right)$ is a regular Dirichlet form on $L^2 (\Omega, dx )$. Note that $\mathcal{D}_{\mathcal{E}} (\Omega)$ is dense in  $L^2(\Omega, dx )$ since $\left( \mathcal{E}, \mathcal{D}_{\mathcal{E}} (\Omega) \right)$ is a regular Dirichlet form. We denote by $-\mathcal{L}_{\Omega}$ the non-negative self-adjoint operator on $L^2\left( \Omega, dx \right)$ corresponding to the Dirichlet form $\left( \mathcal{E}, \mathcal{D}_{\mathcal{E}} (\Omega) \right)$.

\begin{notation} We denote the \emph{semigroups} corresponding to the Dirichlet forms $\left( \mathcal{E}, \mathcal{D}_{\mathcal{E}}\right)$ and $\left( \mathcal{E}, \mathcal{D}_{\mathcal{E}}\left( \Omega \right) \right)$ by
\begin{align*}
& P_{t}:= e^{t \mathcal{L}},
\\
& P_{t}^{\Omega}:= e^{t \mathcal{L}_{\Omega}}
\end{align*}
respectively.
\end{notation}
The domains of the corresponding infinitesimal generators are given by
\begin{align*}
& \mathcal{D}(-\mathcal{L}):= \left\{ f\in L^{2} (\mathbb{G}, dx) : \lim_{t\downarrow 0} \frac{P_{t} f-f}{t} \text{ exists } \right\},
\\
& \mathcal{D}(-\mathcal{L}_{\Omega}):= \left\{ f\in L^{2} (\Omega, dx) : \lim_{t\downarrow 0} \frac{P^{\Omega}_{t} f-f}{t} \text{ exists } \right\}.
\end{align*}
We now give a probabilistic description of the semigroups $P_{t}$ and $P_{t}^{\Omega}$. Homogeneous Carnot groups are complete metric spaces, therefore by \cite[Proposition 3.1]{Sturm1996a} the semigroup $P_{t}$ has a heat kernel $p_{t}\left( x, y \right)$ which is continuous in $x, y \in \mathbb{G}$. We also know that by Chow-Rashevskii's theorem that the intrinsic metric induced by the Dirichlet form $\mathcal{E}$ coincides with the original topology, so the Dirichlet form is \emph{strongly regular}. The heat kernel $p_{t} (x,y)$ is the transition density of $g_{t}$, that is,
\begin{equation}\label{eqn.global.tran.prob}
\Prob^{x} \left(  g_{t} \in E \right) = \int_E p_{t}(x,y)dy,
\end{equation}
for any Borel set $E$ in $\mathbb{G}$, where $g_{t}$ is the Markov process whose generator is the sub-Laplacian $\mathcal{L}$.  The strongly continuous semigroup associated to the Dirichlet form  $\left( \mathcal{E}, \mathcal{D}_{\mathcal{E}}  \right)$ is given in terms of $g_{t}$ by
\begin{equation*}
P_{t} f(x) = \E^{x} [ f(g_{t}) ], \quad x \in \mathbb{G}, \; f \in L^2(\mathbb{G}, dx).
\end{equation*}

\begin{definition}\label{dfn.hypoBM}
Let $g_{t}$ be the $\mathbb{G}$-valued Markov process with the transition density given by the heat kernel $p_{t}(x,y)$. Then we refer to $g_{t}$ as the \emph{hypoelliptic Brownian motion}.
\end{definition}
The heat kernel is the fundamental solution to the heat equation
\begin{align}
& \left( \partial_{t} - \frac{1}{2} \mathcal{L} \right) p_{t}(x, \cdot) =0, \notag
\\
& p_{t} (x,y) dy \rightarrow \text{Dirac}_{x} (dy) \quad \text{weakly as} \; t\rightarrow 0, \label{eqn.delta}
\end{align}
where $\text{Dirac}_{x}(dy)$ is the Dirac measure centered at $\{x\}$, see \cite[Equation (3.6)]{DriverGrossSaloff-Coste2009a}.

Following   \cite[Section 2.11]{ByczkowskiLectureNotes2013}, we let $\Omega$ be a bounded open connected set and
\[
\tau_{\Omega}:= \inf \left\{ t>0 \, : \, g_{t} \notin \Omega \right\}.
\]
Then we can use Hunt's formula
\begin{equation}\label{eqn.Dirichlet.HK}
p^{\Omega}_{t}(x,y):= p_{t}(x,y) - \E^x \left[ \mathbbm{1}_{\{ \tau_{\Omega} < t
\} } \,  p_{ t- \tau_{\Omega}} \left( g_{\tau_{\Omega}}, y\right) \right]
\end{equation}
for the transition density $p^{\Omega}_{t} (x,y)$ of the killed Markov process $g_{t}^{\Omega}$ given by
\begin{equation*}
g_t^{\Omega}:=
\left\{ \begin{array}{cc}
g_t & t < \tau_{\Omega},
\\
\partial & t \geqslant \tau_{\Omega},
\end{array} \right.
\end{equation*}
where $\partial$ is the \emph{cemetery} point. More precisely, we have that
\begin{equation}\label{eqn.density}
\Prob^{x} \left( g^{\Omega}_{t} \in E \right) = \Prob^{x} \left( g_{t} \in E,  \, t<\tau_{\Omega} \right) = \int_{E} p^{\Omega}_{t}(x,y) dy,
\end{equation}
for any $x\in \Omega$ and any Borel subset $E$ of $\Omega$. In particular,
\begin{equation}\label{eqn.exit.time}
\Prob^{x} \left( \tau_{\Omega} > t \right) = \int_{\Omega} p^{\Omega}_{t}(x,y)dy.
\end{equation}
We refer to $p_{t}^{\Omega}$ as the Dirichlet heat kernel. Note that regularity of the Dirichlet form implies that $g^{\Omega}_{t}$ is a Hunt process as well. Then the semigroup $P_{t}^{\Omega}$ can be viewed as
\begin{align}
& P^{\Omega}_{t} : L^{2}(\Omega, dx) \longrightarrow L^{2}(\Omega, dx) ,  \notag
\\
& P^{\Omega}_{t} f (x) = \E^{x} \left[ f(g^{\Omega}_{t} ) \right] =\E^{x} \left[ f(g_{t} ), \, t < \tau_{\Omega} \right]= \int_{\Omega} p_{t}^{\Omega} (x,y) f(y) dy , \label{eqn.semigroup.restricted}
\end{align}
for any $f \in L^{2}(\Omega, dx)$. Note that by \cite[Proposition 3.1]{Sturm1996a} applied to the Dirichlet  form $(\mathcal{E},  \mathcal{D}_{\mathcal{E}} \left( \Omega \right) )$ on $L^{2} (\Omega, dx)$ it follows that the function $p^{\Omega}_{t}(x,y)$ is H\"older continuous on $[T, \infty ) \times \Omega \times \Omega$.

The following theorem is the main result of this section, where we collect spectral properties of the operator $\mathcal{L}_{\Omega}$.

\begin{theorem}\label{thm.spectral.results}
Let $\Omega$ be a bounded open connected subset of $\mathbb{G}$, and $P^{\Omega}_{t}$ the semigroup associated to the Dirichlet form $\left( \mathcal{E}, \mathcal{D}_{\mathcal{E}}\left( \Omega \right) \right)$ with generator $\mathcal{L}_{\Omega}$. Then
\begin{enumerate}
\item For every $t>0$, $P^{\Omega}_{t} $ is a Hilbert-Schmidt operator on $L^{2} (\Omega, dx)$ and its spectrum is given by
\[
\sigma (P^{\Omega}_{t} )\backslash \{ 0 \} = \left\{ e^{-\lambda_{n} t} \right\}_{n \in \mathbb{N}},
\]
where $\{ \lambda_{n} \}_{n\in \mathbb{N}} = \sigma_{pp} (- \mathcal{L}_{\Omega} )$  is the point spectrum of $- \mathcal{L}_{\Omega}$, with $0\leqslant \lambda_{1} \leqslant \lambda_{2} \leqslant \cdots$.

\item The operator $-\mathcal{L}_{\Omega}$ has a spectral gap, that is, $\lambda_{1} >0$.

\item There exists an orthonormal basis $\{ \phi_{n} \}_{n\in \mathbb{N}}$ of $L^{2} (\Omega, dx)$ such that for every $n\in \mathbb{N},\, t>0$,
\begin{align*}
& P_{t}^{\Omega} \phi_{n}  = e^{-\lambda_{n} t} \phi_{n},
\end{align*}
Moreover, for  every $n\in \mathbb{N}$, $\phi_{n} \in \mathcal{D} ( -\mathcal{L}_{\Omega})$ and
\begin{align*}
- \mathcal{L}_{\Omega} \phi_{n} =  \lambda_{n} \phi_{n}.
\end{align*}
\end{enumerate}
\end{theorem}

\begin{corollary}\label{cor.discrete.spec}
The operator $\mathcal{L}_{\Omega}$ has a pure point spectrum.
\end{corollary}
\begin{proof}[Proof of Corollary \ref{cor.discrete.spec}]
By \cite[Theorem XIII.64 p.245]{ReedSimonIV} and Theorem \ref{thm.spectral.results} part (3) it follows that $(-\mu - \mathcal{L}_{\Omega} )^{-1}$ is a compact operator for every $\mu$ in the resolvent set of $\mathcal{L}_{\Omega}$, proving that $\mathcal{L}_{\Omega}$ has a pure point spectrum.
\end{proof}

\begin{proof}[Proof of Theorem \ref{thm.spectral.results}]
(1) Let us first show that the semigroup  $P^{\Omega}_{t}$ is a  Hilbert-Schmidt operator for each $t$. We rely on \cite{Sturm1996a} for a heat kernel estimate which only requires volume doubling and Poincar\'{e}'s inequality without compactness assumption on the underlying metric space, unlike in \cite{JerisonSanchez-Calle1986}. Doubling and Poincar\'{e}'s inequality are known to hold in our setting, fro example, by \cite[Theorem 3.2]{DriverGrossSaloff-Coste2009a}. By \cite[Equation (4.2)]{Sturm1996a} with $\varepsilon=1$, there exists a constant $A>0$ such that
\begin{align}\label{eqn.bound.heat.kernel}
& p_{t} (x,y) \leqslant A t^{-\frac{Q}{2}},
\end{align}
for all $t>0$ and $x,y\in \mathbb{G}$, where $Q$ is the homogeneous dimension of $\mathbb{G}$. Thus, for any $t>0$
\begin{align*}
&\Vert p^{\Omega}_{t} \Vert^{2}_{L^{2} (\Omega \times \Omega)}= \int_{\Omega} \int_{\Omega} p^{\Omega}_{t} (x,y)^{2} dx dy  \leqslant  \int_{\Omega} \int_{\Omega} p_{t}  (x,y)^{2} dx dy
\\
&\leqslant   A^{2} \vert \Omega \vert ^2t^{-Q}< \infty,
\end{align*}
where we used that $p_{t}^{\Omega}(x,y) \leqslant p_{t} (x,y)$ for almost all $x,y \in \Omega$. Indeed, for any Borel set $E\subset \Omega$, $t>0$, and all $x\in \Omega$ we have that
\begin{align*}
& \int_{E} p^{\Omega}_{t}(x,y)dy= \Prob^{x} \left( g_{t}^{\Omega}\in E \right) = \Prob^{x} \left( g_{t} \in E, \tau_{\Omega} >t \right)
\\
& \leqslant \Prob^{x} \left( g_{t}\in E \right) = \int_{E} p_{t}(x,y)dy,
\end{align*}
and hence $p^{\Omega}_{t} (x,y) \leqslant p_{t} (x,y)$ for any $x\in \Omega$ and for a.e. $y\in\Omega$. The estimate then follows for every $y\in \Omega$ since both $p_{t}$ and $p_{t}^{\Omega}$ are continuous on $\Omega\times \Omega$.

The operator $P^{\Omega}_{t}$ is then Hilbert-Schimdt since $p_{t}^{\Omega} \in L^{2} (\Omega \times \Omega)$.  In particular,  for every $t>0$ the operator $P^{\Omega}_{t}$ is compact, and hence by the spectral theorem for compact operators there exists a sequence of decreasing eigenvalues  $\{  \lambda_{n} (t)\}_{n\in \mathbb{N}}$  and corresponding eigenfunctions $\{ \phi_{n,k}^{(t)} \}_{n,k \in \mathbb{N}}$ such that  for every $t>0$
\begin{align}
& \sigma (P^{\Omega}_{t} ) \backslash \{ 0 \} = \sigma_{pp} (P^{\Omega}_{t} ) \backslash \{ 0 \} = \{  \lambda_{n} (t)\}_{n\in \mathbb{N}},  \notag
\\
&  \ker \left( \lambda_{n} (t) - P^{\Omega}_{t}  \right)= \overline{\operatorname{Span}  \{ \phi_{n,k}^{(t)}, k\in \mathbb{N} \} } \text{ for every } n\in \mathbb{N},  \notag
\\
& L^{2}(\Omega, dx)=\bigoplus_{n=1}^{\infty} \overline{\operatorname{Span} \{ \phi_{n,k}^{(t)}, k \in \mathbb{N}  \}}.\label{eqn.orth.bas}
\end{align}
The semigroup  $P_{t}^{\Omega}$ is strongly continuous because the Dirichlet form $\left( \mathcal{E}, \mathcal{D}_{\mathcal{E}} \left( \Omega \right) \right)$ is regular, and hence  by the spectral mapping theorem for semigroups \cite[Theorem 6.3]{ArendtGraboschGreinerMoustakasNagelNeubranderSchlotterbeck1986}
\begin{align}\label{eqn.spec.mapp.thm}
\sigma_{pp}  (  P^{\Omega}_{t}  ) \backslash \{0\} = \exp \left( t\,  \sigma_{pp} ( \mathcal{L}_{\Omega} ) \right), \text{ for any }  t>0.
\end{align}
Thus the eigenvalues of $P_{t}^{\Omega}$ are given by $e^{\mu_{n} t}$ for $\mu_{n} \in \sigma_{pp} ( \mathcal{L}_{\Omega} )$. By the theory of Dirichlet forms \cite{FukushimaOshimaTakedaBook2011}, the operator $\mathcal{L}_{\Omega}$ is non-positive definite and hence we can write  $\mu_{n} = - \lambda_{n}$, where $\lambda_{n} \rightarrow \infty$ as $n\rightarrow \infty$, which completes the proof of (1).

(2) Let us now prove that $\lambda_{1}>0$. Assume that $\lambda_{1}=0$, then by the spectral mapping theorem \eqref{eqn.spec.mapp.thm} we have that $1\in \sigma_{pp} \left( P_{t}^{\Omega}\right)$ for all $t>0$, and hence
\begin{align*}
1\leqslant \Vert P_{t}^{\Omega} \Vert,
\end{align*}
since $\sigma_{pp} (P_{t}^{\Omega}  )  \subset \{ z\in \mathbb{C}, \; \vert z \vert \leqslant \Vert P^{\Omega}_{t} \Vert \}$. By \eqref{eqn.bound.heat.kernel} we have that
\begin{align*}
& 1 \leqslant \Vert P^{\Omega}_{t} \Vert^{2} \leqslant \Vert p^{\Omega}_{t} \Vert^{2}_{L^{2} (\Omega \times \Omega)} \leqslant A^{2} \vert \Omega \vert^{2} t^{-Q},
\end{align*}
for some finite constant $A>0$. Thus, $t^{Q} \leqslant  A^{2} \vert \Omega \vert^{2}$ for any $t>0$, which is a contradiction.

(3) Let $-\lambda_{n} \in \sigma_{pp} ( \mathcal{L}_{\Omega} )$, and let $\{ \phi_{n,k} \}_{n,k \in \mathbb{N}}$ be an orthonormal basis of $\ker \left( -\lambda_{n}  - \mathcal{L}_{\Omega} \right)$. By \cite[Corollary 6.4]{ArendtGraboschGreinerMoustakasNagelNeubranderSchlotterbeck1986} it follows that
\begin{align*}
& \overline{\operatorname{Span}\{ \phi_{n,k}^{(t)}, k \in \mathbb{N} \}} = \ker \left( e^{-\lambda_{n} t}-P^{\Omega}_{t} \right)
\\
& = \overline{\operatorname{Span} \left\{ \ker \left( -\lambda_{n}  + \frac{2 \pi j}{t} i - \mathcal{L}_{\Omega} \right), j \in \mathbb{Z}\right\}}
\end{align*}
for any  $t>0$.
The point spectrum of  $\mathcal{L}_{\Omega}$ is real since $\mathcal{L}_{\Omega}$ is self-adjoint, and hence
\[
\ker \left( -\lambda_{n}  + \frac{2 \pi j}{t} i - \mathcal{L}_{\Omega} \right)= \{ 0 \}
\]
for all $j\not= 0$. Thus we have that
\begin{align}
& \overline{ \operatorname{Span}\{ \phi_{n,k}^{(t)}, k\in\mathbb{N} \}} = \ker \left( e^{-\lambda_{n} t}-P^{\Omega}_{t} \right)   \notag
\\
& = \ker \left( -\lambda_{n}   - \mathcal{L}_{\Omega} \right)= \overline{ \operatorname{Span} \{ \phi_{n,k},  k\in\mathbb{N}\}}, \text{ for all }  n\in \mathbb{N}. \notag
\end{align}
The operator $P^{\Omega}_{t}$ is compact, and thus for every $n\in\mathbb{N}$ the eigenspace $\ker \left( e^{-\lambda_{n} t}  - P^{\Omega}_{t} \right)$ is finite-dimensional. Therefore for every $n\in \mathbb{N}$ there exists an $M_{n}$ such that for every $t>0$

\begin{equation}\label{eqn.orth.bas2}
\ker \left( e^{-\lambda_{n} t}  - P^{\Omega}_{t} \right)  = \operatorname{Span} \{ \phi_{n,k}, k=1,\ldots, M_{n} \}.
\end{equation}
We proved that for every $n\in \mathbb{N}$ there exists an orthonormal basis $\{ \phi_{n,k} \}_{k=1}^{M_{n}}$ of $\ker \left( e^{-\lambda_{n} t}  - P^{\Omega}_{t} \right)$ for any $t>0$ such that for any $k=1,\ldots,  M_{n}$
\begin{align}
& P_{t}^{\Omega} \phi_{n,k} = e^{-\lambda_{n} t} \phi_{n,k},   \notag
\\
& \mathcal{L}_{\Omega} \phi_{n,k} = -\lambda_{n} \phi_{n,k} . \notag
\end{align}
By \eqref{eqn.orth.bas} and \eqref{eqn.orth.bas2} it follows that
\begin{align}\label{eqn.decomposition}
& L^{2}(\Omega, dx) =\bigoplus_{n=1}^{\infty} \operatorname{Span} \{ \phi_{n,k}, k=1,\ldots,M_{n} \},
\end{align}
and hence $\{\phi_{n,k} \}_{k=1, n=1}^{M_{n}, \infty}$ is the desired orthonormal basis of $L^{2}(\Omega, dx)$.
\end{proof}

\begin{notation}\label{notation.basis}
Throughout the paper instead of using the orthonormal basis $\{\phi_{n,k} \}_{k=1, n=1}^{M_{n}, \infty}$ given in the proof of Theorem \ref{thm.spectral.results}, we denote by $\{\phi_{n}\}_{n=1}^{\infty}$ the same orthonormal basis, where for each repeated eigenvalue $\lambda_{n}$ we index the corresponding eigenfunctions consequently according to its (finite) multiplicity $M_{n}$. In particular,  we have that

\begin{equation}\label{eqn.decomp.semigroup}
P^{\Omega}_{t} f = \sum_{n=1}^{\infty} e^{-\lambda_{n} t} \sum_{k=1}^{M_{n}} \langle f, \phi_{n,k} \rangle \phi_{n,k}= \sum_{n=1}^{\infty} e^{-\lambda_{n} t} \langle f, \phi_{n} \rangle \phi_{n}
\end{equation}
for every $f \in L^{2}(\Omega, dx)$.
\end{notation}
We next prove regularity properties of the eigenfunctions of $-\mathcal{L}_{\Omega}$.

\begin{proposition}\label{prop.efunctions.properties}
Let $\Omega$ be a bounded open connected subset of $\mathbb{G}$, and $\{ \phi_{n} \}_{n=1}^{\infty}$ be the eigensystem of $-\mathcal{L}_{\Omega}$ with eigenvalues $\{ \lambda_{n} \}_{n=1}^{\infty}$. Then
\begin{enumerate}
\item There exists a constant $d(\Omega)$ such that for any $1\leqslant p \leqslant \infty$,
\[
\Vert \phi_{n} \Vert_{L^{p} (\Omega, dx)} \leqslant d(\Omega) \lambda_{n}^{\frac{Q}{2}},
\]
where $Q$ is the homogeneous dimension of $\mathbb{G}$.

\item For every $n \in \mathbb{N}$, the function $\phi_{n}$ is continuous in $\Omega$.

\item The series
\begin{align*}
\sum_{n=1}^{\infty} e^{-\lambda_{n} t} \phi_{n} (x) \phi_{n}(y),
\end{align*}
converges uniformly on $\Omega \times \Omega \times [\varepsilon, \infty)$, for any $\varepsilon>0$.

\item If $\Omega$ is regular in the sense of Definition \ref{d.regular.set.prob}, then for every $n\in \mathbb{N}$ and $z\in \partial \Omega$
\begin{equation}\label{eqn.boundary}
\lim_{x \rightarrow z}  \phi_{n} (x)  =0.
\end{equation}

\end{enumerate}
\end{proposition}

\begin{proof}
(1) First, note that $P^{\Omega}_{t}: L^{2}(\Omega, dx) \rightarrow L^{2}(\Omega, dx)$ is a self-adjoint operator since its generator $\mathcal{L}_{\Omega}$ is self-adjoint. Thus
\begin{equation}\label{eqn.1a}
(P^{\Omega}_{t})^{\ast} f = P^{\Omega}_{t} f, \, \text{for any}\, f\in L^{2}(\Omega, dx),
\end{equation}
where $(P^{\Omega}_{t})^{\ast}$ denotes the adjoint of $P^{\Omega}_{t}$. By \eqref{eqn.bound.heat.kernel} there exists a constant $A>0$ such that, for any $x,y\in \Omega$
\[
p^{\Omega}_{t} (x,y) \leqslant p_{t} (x,y) \leqslant  At^{-\frac{Q}{2}}.
\]
Therefore, for any $1\leqslant p \leqslant \infty$ and any $x\in \Omega$
\[
\Vert p_{t}^{\Omega} (x, \cdot ) \Vert^{p}_{L^{p}(\Omega)}= \int_{\Omega} p^{\Omega}_{t} (x,y)^{p} dy\leqslant \vert \Omega \vert A^{p} t^{-\frac{Q}{2}p},
\]
that is,
\[
\Vert p^{\Omega}_{t} (x,\cdot) \Vert_{L^{p}(\Omega)} \leqslant \vert \Omega \vert^{\frac{1}{p}} A t^{-\frac{Q}{2}}.
\]
Therefore for any $1\leqslant q\leqslant \infty$ and $f\in L^{q}(\Omega,dx)$
\begin{align*}
& \Vert P^{\Omega}_{t} f \Vert^{2}_{L^{2}(\Omega)} = \int_{\Omega} \left[ \int_{\Omega} f(y) p^{\Omega}_{t} (x,y)dy \right]^{2} dx
\\
& \leqslant \int_{\Omega}  \Vert f \Vert^{2}_{L^{q}(\Omega) } \Vert p^{\Omega}_{t} (x, \cdot ) \Vert^{2}_{L^{\frac{q}{q-1}}(\Omega)}dx\leqslant  \Vert f \Vert^{2}_{L^{q}(\Omega) } \vert \Omega\vert^{1+ \frac{2q-2}{q}} A^{2} t^{-Q}.
\end{align*}
Thus the operator $P^{\Omega}_{t} : L^{q}(\Omega,dx) \rightarrow  L^{2}(\Omega, dx)$ is well-defined for every $1\leqslant q \leqslant \infty$, and it satisfies
\begin{equation}\label{eqn.2a}
\Vert P^{\Omega}_{t} \Vert_{L^{q}(\Omega)\to L^{2}(\Omega)} \leqslant c(q,\Omega) t^{-\frac{Q}{2}},
\end{equation}
where $c(q,\Omega):= \vert \Omega \vert^{\frac{1}{2} + \frac{q-1}{q}}A \leqslant A \max (\vert \Omega \vert^{\frac{1}{2}}, \vert \Omega\vert^{\frac{3}{2}} ) =: c(\Omega)$. The adjoint $(P^{\Omega}_{t})^{\ast}  :  L^{2}(\Omega,dx) \rightarrow  L^{p}(\Omega,dx)$ then satisfies
\[
\Vert (P^{\Omega}_{t})^{\ast} \Vert_{L^{2}(\Omega)\to L^{p}(\Omega)} \leqslant c(\Omega) t^{-\frac{Q}{2}},
\]
where $p$ is the conjugate of $q$. Thus, by \eqref{eqn.1a}
\[
\Vert P^{\Omega}_{t} \Vert_{L^{2}(\Omega)\to L^{p}(\Omega)} \leqslant c(\Omega) t^{-\frac{Q}{2}},
\]
for any $1\leqslant p \leqslant \infty$.

Let $\phi_{n}$ be an eigenfunction for $P^{\Omega}_{t}$ with the eigenvalue $e^{-\lambda_{n} t}$, then it follows that
\begin{align*}
\Vert \phi_{n} \Vert_{L^{p} (\Omega) }\leqslant c(\Omega) t^{-\frac{Q}{2}} e^{\lambda_{n} t} \Vert \phi_{n} \Vert_{ L^{2}(\Omega)},
\end{align*}
and taking the infimum over $t>0$ we see that
\begin{align*}
\Vert \phi_{n} \Vert_{L^{p} (\Omega) }\leqslant c(\Omega) \left(  \frac{2e}{Q}\right)^{\frac{Q}{2}} \Vert \phi_{n} \Vert_{ L^{2}(\Omega)} \lambda_{n}^{\frac{Q}{2}}.
\end{align*}
(2) Note that, for any $f\in L^{2} (\Omega, dx)$, the function
\begin{align*}
x \longrightarrow \int_{\Omega} f(y) p^{\Omega}_{t} (x,y) dy
\end{align*}
is continuous in $x$ since $p_{t}^{\Omega}(x,y)$ is continuous in $x$ and $y$ in $\Omega$. Then
\begin{align*}
\phi_{n}  (x) = e^{\lambda_{n} t} (P^{\Omega}_{t} \phi_{n} ) (x)=e^{\lambda_{n} t} \int_{\Omega} \phi_{n}(y) p^{\Omega}_{t} (x,y) dy
\end{align*}
is continuous  for any $x \in \Omega$.

(3) Let $\varepsilon>0$. Then for every $x,y\in \Omega$, $t\geqslant \varepsilon$ we have that
\begin{align*}
& \sum_{n=1}^{\infty} e^{-\lambda_{n} t} \vert \phi_{n} (x) \phi_{n}(y) \vert \leqslant\sum_{n=1}^{\infty} e^{-\lambda_{n} t} \Vert \phi_{n} \Vert_{  L^{\infty} (\Omega, dx ) }^{2}
\\
& \leqslant d(\Omega)^{2}\sum_{n=1}^{\infty} e^{-\lambda_{n} t}\lambda_{n}^{Q} \leqslant  d(\Omega)^{2} \sum_{n=1}^{\infty}   e^{-\lambda_{n} \varepsilon} \lambda_{n}^{Q},
\end{align*}
which is convergent.

(4) For any $t>0$,  $n \in \mathbb{N}$, and $x\in \Omega$  we have that
\begin{align*}
& \vert \phi_{n} (x) \vert = e^{\lambda_{n}t} \left| \int_{\Omega} p^{\Omega}_{t} (x,y) \phi_{n} (y)dy \right| \leqslant e^{\lambda_{n}t} \Vert \phi_{n} \Vert_{L^{\infty} (\Omega, dx)} \int_{\Omega} p^{\Omega}_{t} (x,y) dy
\\
& =  e^{\lambda_{n}t} \Vert \phi_{n} \Vert_{L^{\infty} (\Omega, dx)} \Prob^{x} \left( \tau_{\Omega} >t \right).
\end{align*}
By \cite[Proposition 1, p. 163]{ChungWalshBook} we have that the function $x\rightarrow  \Prob^{x} \left( \tau_{\Omega} >t \right)$ is upper semi-continuous for any $x\in\mathbb{G}$. Though their proof is for a standard Brownian motion, it only relies on the semigroup property, an thus the argument applies in our setting. If $\Omega$ is regular, then for any $z\in \partial \Omega$ we have that
\begin{align*}
& \lim_{x\rightarrow z} \vert \phi_{n} (x) \vert \leqslant e^{\lambda_{n}t}  \Vert \phi_{n} \Vert_{L^{\infty} (\Omega, dx)}        \limsup_{x\rightarrow z} \Prob^{x} \left( \tau_{\Omega} >t \right)
\\
& \leqslant e^{\lambda_{n}t}  \Vert \phi_{n} \Vert_{L^{\infty} (\Omega, dx)}         \Prob^{z} \left( \tau_{\Omega} >t \right)
\\
& \leqslant   e^{\lambda_{n}t}  \Vert \phi_{n} \Vert_{L^{\infty} (\Omega, dx)}         \Prob^{z} \left( \tau_{\Omega} >0 \right)  =0,
\end{align*}
\end{proof}

Regularity of $\Omega$ was only used in the proof of Part (4) of Proposition~\ref{prop.efunctions.properties} to check that the eigenfunctions vanish on $\partial \Omega$. For the rest of this section we do not assume regularity of the set $\Omega$.

\begin{corollary}\label{cor.important} Let $x, y \in \Omega$, and $t>0$, then
\begin{align*}
& p^{\Omega}_{t} (x,y) = \sum_{n=1}^{\infty} e^{-\lambda_{n} t}  \phi_{n} (x) \phi_{n}(y),
\\
& \Prob^{x} \left( \tau_{\Omega} >t \right) = \sum_{n=1}^{\infty} e^{-\lambda_{n} t}   c_{n} \phi_{n} (x),
\end{align*}
where $c_{n}:= \int_{\Omega} \phi_{n} (y) dy$.
\end{corollary}

\begin{proof}
By \eqref{eqn.decomp.semigroup} and Proposition \ref{prop.efunctions.properties} part (3) we have that
\begin{align*}
& \int_{\Omega} f(y) p^{\Omega}_{t} (x,y) dy= P^{\Omega}_{t} f (x) = \sum_{n=1}^{\infty} e^{-\lambda_{n}t} \langle f, \phi_{n} \rangle_{L^{2} (\Omega, dx )} \phi_{n}(x)
\\
& =  \sum_{n=1}^{\infty} e^{-\lambda_{n}t}  \int_{\Omega} f(y) \phi_{n} (y) dy\phi_{n}(x) =  \int_{\Omega} f(y)  \sum_{n=1}^{\infty}  e^{-\lambda_{n}t}   \phi_{n} (y) \phi_{n}(x) dy,
\end{align*}
for any $f \in L^{2} (\Omega, dx )$, and hence  $p^{\Omega}_{t} (x,y) = \sum_{n=1}^{\infty} e^{-\lambda_{n} t}  \phi_{n} (x) \phi_{n}(y)$. Then
\begin{align*}
&  \Prob^{x} \left( \tau_{\Omega} >t \right)  = \int_{\Omega} p^{\Omega}_{t} (x,y) dy = \int_{\Omega} \sum_{n=1}^{\infty} e^{-\lambda_{n} t}  \phi_{n} (x) \phi_{n}(y) dy
\\
& = \sum_{n=1}^{\infty} e^{-\lambda_{n} t}   \phi_{n} (x) \int_{\Omega}  \phi_{n}(y) dy = \sum_{n=1}^{\infty} e^{-\lambda_{n} t} c_{n} \phi_{n} (x) ,
\end{align*}
where $c_{n} := \int_{\Omega} \phi_{n} (y)dy$ is finite by Proposition \ref{prop.efunctions.properties} part (1).
\end{proof}

Our next goal is to prove that the first eigenvalue $\lambda_{1}$ is a simple eigenvalue for $- \mathcal{L}_{\Omega}$ or equivalently, by Theorem \ref{thm.spectral.results}, that $e^{-\lambda_{1} t}$ is a simple eigenvalue for $P^{\Omega}_{t}$. This will follow from the irreducibility of the semigroup $P^{\Omega}_{t}$. The definition of irreducibility of Dirichlet forms and corresponding semigroups can be found in \cite[p. 55]{FukushimaOshimaTakedaBook2011}. For a definition of irreducible semigroups on Banach lattices we refer to \cite[Section 14.3]{BatkaiKramarRhandiBook2017}. We will use the following characterization of irreducible semigroups \cite[Example 14.11]{BatkaiKramarRhandiBook2017}. Let $T_{t}$ be a strongly continuous semigroup on $L^{p} (\Omega, dx)$, $1\leqslant p < \infty$ with generator $A$. Let $s(A):= \sup \{ \text{Re} (\lambda), \lambda \in \sigma(A) \}$ and $R_{\mu} = (A-\mu )^{-1}$ for $\mu$ in the resolvent set of $A$.

\begin{lemma}[Example 14.11 in \cite{BatkaiKramarRhandiBook2017}]\label{lemma.irredu}
The semigroup $T_{t}$ is irreducible if and only if for any positive $f \in  L^{p} (\Omega, dx)$ we have that
\begin{align*}
R_{\mu}f (x) >0, \text{ for a.e. }  x\in \Omega \text{ and some }  \mu >s(A).
\end{align*}
\end{lemma}

\begin{theorem}\label{rmk.irreduc}
The semigroup $P^{\Omega}_{t}$ is irreducible.
\end{theorem}

\begin{proof}
We first prove that
\begin{equation}\label{eqn.positivity}
p_{t}^{\Omega} (x,y)>0,
\end{equation}
for every $t>0$ and $x, y \in \Omega$. We claim that for every $y\in \Omega$ and $r$ small enough, there exists a time $t_{0}$ such that for any $x\in B_{r} (y), \; z\in \partial \Omega$ and $s<t< t_{0}$ one has that
\begin{align}\label{eqn.kinda.there}
p_{t}(x,y) - p_{s} (z,y) >0.
\end{align}
Indeed, if we assume \eqref{eqn.kinda.there} then
\begin{align*}
& p^{\Omega}_{t}(x,y):= p_{t}(x,y) - \E^x \left[ \mathbbm{1}_{\{ \tau_{\Omega} < t
\} } \,  p_{ t- \tau_{\Omega}} \left( g_{\tau_{\Omega}}, y\right) \right]
\\
& \geqslant p_{t}(x,y) - \E^x \left[  p_{ t- \tau_{\Omega}} \left( g_{\tau_{\Omega}}, y\right) \right] = \E^x \left[ p_{t}(x,y) - p_{ t- \tau_{\Omega}} \left( g_{\tau_{\Omega}}, y\right) \right] >0
\end{align*}
for any $t<t_{0}$ and $x\in B_{r}(y)$. The result would then follow for any $x\in \Omega$ by a standard chaining argument.  Let us now prove \eqref{eqn.kinda.there}. By \cite[Equation  (3.7)]{DriverGrossSaloff-Coste2009a}, for any $k\in (0,1)$ there exists a $c_{k} \in (0,\infty )$ such that
\begin{align*}
p_{s} (z,y) \leqslant c_{k} \left( 1+ \frac{1}{s}\right)^{\frac{\theta}{2}} e^{c_{k} s} e^{- k\frac{ d(z,y)^{2} }{s} } \leqslant c_{k} \left( 1+ \frac{1}{s}\right)^{\frac{\theta}{2}} e^{c_{k} s} e^{-k \frac{ d(y, \partial\Omega)^{2} }{s} },
\end{align*}
where the last inequality follows from the fact that $d(y,\partial \Omega)\leqslant d (z,y)$ since $z\in \partial \Omega$. By \cite[Equation (3.8)]{DriverGrossSaloff-Coste2009a} there exist constants $c_{1}, c_{2} \in (0,\infty)$ such that
\begin{align*}
p_{t} (x,y) \geqslant c_{1} \left( 1+ \frac{1}{t}\right)^{\frac{\theta}{2}} e^{-c_{2} t} e^{- c_{2}\frac{ d(x,y)^{2} }{t} } ,
\end{align*}
where $\theta$ is an integer defined in \cite[eq. (3.4)]{DriverGrossSaloff-Coste2009a}. For the sake of conciseness set
\begin{align*}
& u:= 1+ \frac{1}{t}, \qquad  v:= 1+ \frac{1}{s},
\\
& \alpha := \frac{\theta}{2}, \qquad \beta: = c_{2} d(x,y)^{2}, \qquad \gamma := k d(y, \partial\Omega)^{2}.
\end{align*}
Then
\begin{align*}
& p_{t} (x,y) - p_{s} (z,y) \geqslant c_{1} u^{\alpha} e^{-c_{2} t} e^{-\frac{\beta}{t}} - c_{k} v^{\alpha} e^{c_{k} s} e^{- \frac{\gamma}{s} }
\\
& = c_{k} v^{\alpha} e^{c_{k} s} e^{- \frac{\gamma}{s} } \left( \frac{c_{1}}{c_{k}}  \frac{u^{\alpha}}{v^{\alpha}} \frac{e^{-c_{2}t}}{e^{c_{k}s}} \frac{ e^{- \frac{\beta}{t} } }{  e^{- \frac{\gamma}{s} }}    -1 \right),
\end{align*}
and hence it is enough to prove that
\begin{align*}
  \frac{u^{\alpha}}{v^{\alpha}} \frac{ e^{- \frac{\beta}{t} } }{  e^{- \frac{\gamma}{s} }}    > \frac{c_{k} e^{c_{k} s}}{c_{1} e^{-c_{2} t}},
\end{align*}
for all $0<s<t$ small enough, that is,
\begin{align}\label{eqn.parameters}
  \frac{u^{\alpha}}{v^{\alpha}} e^{- \beta u + \gamma v + \beta - \gamma }> \frac{c_{k} e^{c_{k} s}}{c_{1} e^{-c_{2} t}},
\end{align}
for all $0<s<t$ small enough. If we let $F(v):=  \frac{u^{\alpha}}{v^{\alpha}} e^{- \beta u + \gamma v + \beta - \gamma }$ for $v>u>0$ and $u$ fixed, then
\begin{align*}
F^{\prime}(v) = F(v) \left( \gamma - \frac{\alpha}{v} \right)>0
\end{align*}
for $v$ large enough, and
\begin{align*}
F(u) = e^{ \frac{\gamma - \beta }{t}},
\end{align*}
Thus, if we choose $r$ small enough so that $\gamma - \beta >0$, then  we can find a $t_{0} = t_{0} (x,y, c_{1}, c_{k}, \Omega)$ such that \eqref{eqn.parameters} is satisfied and the proof of  \eqref{eqn.positivity} is complete.

We can now prove the irreducibility of $P^{\Omega}_{t}$. We can use \cite[Exercise 1.3.1]{FukushimaOshimaTakedaBook2011} to express the resolvent $R_{\mu}$ in terms of the semigroup $P^{\Omega}$. Then, by Lemma \ref{lemma.irredu} and \eqref{eqn.positivity} and for any $\mu \in \R$, for any $f>0$, and for a.e. $x \in \Omega$ we have that
\begin{align*}
R_{\mu}f (x) = \int_{0}^{\infty} e^{-\mu t} (P^{\Omega}_{t} f)(x)dx=0
\end{align*}
if and only if  $(P^{\Omega}_{t} f)(x)=0$ for a.e. $t>0$, since $ P^{\Omega}_{t}$ is a positive operator. Thus, $R_{\mu}f (x) =0$ if and only if
\begin{align*}
\int_{\Omega} f(y) p^{\Omega}_{t} (x,y) dy=0,
\end{align*}
that is, if and only if for a.~e. $x\in \Omega$, $p^{\Omega}_{t} (x,y)$ is zero on a set of positive Haar measure, which is not possible by \eqref{eqn.positivity}.
\end{proof}

\begin{theorem}\label{thm.simple.evalue}
Let $\lambda_{1}$ be the first non-zero eigenvalue of $-\mathcal{L}_{\Omega}$. Then $\lambda_{1}$ is a simple eigenvalue and there exists a corresponding eigenfunction $\phi$ such that $\phi (x) >0$ for every $x\in \Omega$.
\end{theorem}

\begin{proof}
For every $t>0$ the operator $P^{\Omega}_{t}$ is compact with spectral radius given by $e^{-\lambda_{1} t}$, and $K:= \left\{ f\in L^{2} (\Omega, dx) f \geqslant 0  \text{ a.s.} \right\}$ is a cone in $L^{2} (\Omega, dx)$ such that $P^{\Omega}_{t} (K) \subset K$. Thus, by Krein-Rutman Theorem \cite{KreinRutman1948}, there exists an eigenfunction $\phi$ of  $P^{\Omega}_{t}$ with eigenvalue $e^{-\lambda_{1} t}$ such that $\phi \in K \backslash \{ 0\}$. By Theorem \ref{thm.spectral.results} we know that $\phi$ is an eigenfunction of $- \mathcal{L}_{\Omega}$ with eigenvalue $\lambda_{1}$. Let us assume that $\phi(x)=0$ for some $x\in \Omega$. Then
\begin{align*}
& 0= \phi(x) = e^{\lambda_{1} t} \int_{\Omega} \phi(y)p^{\Omega}_{t} (x,y)dy \geqslant 0,
\end{align*}
and hence $\phi(y)p^{\Omega}_{t} (x,y)=0$ for a.e. $y\in \Omega$. The set
\begin{align*}
A:= \{ z\in \Omega:  \phi(z) >0 \}
\end{align*}
has positive Haar measure since $\phi\in K \backslash \{0\}$. Thus, $p^{\Omega}_{t} (x,y)=0$ for almost every $y\in A$, which leads to a contradiction by \eqref{eqn.positivity}.

The semigroup $P^{\Omega}_{t}$ is irreducible by Theorem~\ref{rmk.irreduc}, and its generator $\mathcal{L}_{\Omega}$ is self-adjoint, and we proved that there exists $\phi \in \ker (- \lambda_{1} - \mathcal{L}_{\Omega} )$ such that $\phi>0$. Thus, by \cite[Proposition 14.42 (c)]{BatkaiKramarRhandiBook2017} it follows that $\operatorname{dim} \ker (- \lambda_{1} - \mathcal{L}_{\Omega} )=1$.
\end{proof}

\section{Regular boundary points: an analytic approach}\label{sec.regularpts}
In this section we compare the probabilistic notion of regular points in Definition \ref{d.regular.set.prob} with an analytic definition used for hypoelliptic operators. The main goal is to prove that these two notions are indeed equivalent.

Let $\mathcal{L}$ be a diffusion operator  and $\Omega$ be a bounded open connected subset of a homogeneous Carnot group $\mathbb{G}\cong \mathbb{R}^{N}$. Consider the boundary value problem

\[
\left\{
\begin{array}{ll}
\mathcal{L} u =0 & \text{ in } \Omega,
\\
u= \phi & \text{ in } \partial \Omega,
\end{array}
\right.
\]
where $\phi :\partial \Omega \longrightarrow \mathbb{R}$ is a continuous function. If $\Omega$ is an open set with compact closure and nonempty boundary, then there exists a generalized solution $H^{\Omega}_{\phi}$ in the sense of Perron–Wiener–Brelot, which in this setting is described in \cite[II.6.7, p.359]{BonfiglioliLanconelliUguzzoniBook}. We now recall an analytic definition of regular points which can be found in \cite[II.7.11]{BonfiglioliLanconelliUguzzoniBook}.

\begin{definition}\label{d.regular.set.an}
A point $x\in \partial \Omega$ is called \emph{regular} (or $\mathcal{L}$-\emph{regular}) if
\[
\lim_{\Omega \ni z \rightarrow x} H^{\Omega}_{\phi} (z) = \phi \left( x \right)
\]
for every  continuous function $\phi: \partial \Omega \longrightarrow \mathbb{R}$. We call the set $\Omega$ \emph{regular} (or $\mathcal{L}$-\emph{regular})  if every boundary point of $\Omega$ is regular.
\end{definition}

The notion of regular points depends on the operator. The Euclidean space $\R^{N}$ is an example of a homogeneous Carnot group with respect to the Euclidean dilation, and the corresponding differential operator is the standard Laplacian $\Delta_{\R^N}$. If $\Omega$ is any bounded domain in $\R^{N}$ with a $C^{2}$-smooth boundary, then $\Omega$ is $\Delta_{\R^N}$-regular in the sense of Definition \ref{d.regular.set.an} since it satisfies the exterior ball condition \cite[Proposition 7.1.5]{BonfiglioliLanconelliUguzzoniBook}. In \cite{HansenHueber1987} it was shown that this is not true for more general Carnot groups. In particular, there are sub-Laplacians $\mathcal{L}$ on Carnot groups and bounded convex domains with smooth boundary which are not $\mathcal{L}$-regular. Nonetheless, given a Carnot group it is always possible to construct nice regular domains. More precisely, in \cite[Proposition 7.2.8]{BonfiglioliLanconelliUguzzoniBook} it is shown that on a homogeneous Carnot group $\mathbb{G}$ the balls $B_{r}\left( x \right), r>0, x \in \mathbb{G},$ with respect to the $\mathcal{L}$-gauge are regular in the sense of Definition \ref{d.regular.set.an}.

\begin{example}[Heisenberg group]
Suppose $\Hei$ is the Heisenberg group with the group operation given by
\begin{align*}
& (x_1, x_2, x_3) \star (y_1, y_2, y_3)
\\
& := \left(x_1 + y_1, x_2 + y_2,  x_3 + y_3 + \frac{1}{2} (x_1 y_2 - x_2 y_1) \right),
\end{align*}
then by \cite[Example 5.4.7]{BonfiglioliLanconelliUguzzoniBook} the $\mathcal{L}$-gauge is given by \[
\vert x \vert:= \sqrt[4]{ (x_1^2 + x_2^2)^2 + 16 x^2_3}.
\]

We can endow $\Hei$ with a different homogeneous norm by
\[
\rho( x ):= \sqrt[4]{ (x_1^2 + x_2^2)^2 + x^2_3}
\]
and denote by  $B_{r}$ the corresponding ball of radius $r$ centered at the identity. Then in \cite{Gaveau1977a} it is shown that $B_{r}$ is a regular set in the sense of Definition \ref{d.regular.set.prob}.
\end{example}

Let us recall some notions from potential theory \cite[Chapter 7]{BonfiglioliLanconelliUguzzoniBook}. If $V$ is $\mathcal{L}$-regular, then for every fixed $x\in V$ the map
\begin{align*}
& C \left( \partial V, \R \right) \longrightarrow \R
\\
& \phi \longmapsto H^{V}_{\phi} (x)
\end{align*}
is a linear positive functional on $C \left( \partial V, \R \right)$, and hence by the Riesz representation theorem there exists a Radon measure $\mu^{V}_{x}$ supported on $\partial V$ such that
\[
H_{\phi}^{V} (x) = \int_{\partial V} \phi (y) d\mu^{V}_{x} (y).
\]
The measure $\mu^{V}_{x}$ is called the $\mathcal{L}$-\textit{harmonic measure} related to $V$ and $x$.

\begin{definition}\label{def.superharmonic}
Let $\Omega$ be an bounded open connected set. A function $u: \Omega \rightarrow (-\infty, + \infty ]$ is called \emph{$\mathcal{L}$-superharmonic} in $\Omega$ if
\begin{enumerate}
\item $u$ is lower semi-continuous and $u<\infty$ in a dense subset of $\Omega$.
\item for every $\mathcal{L}$-regular open set $V$ with $\overline{V} \subset \Omega$ and for ever $x\in V$
\[
u(x) \geqslant \int_{\partial V} u(y) d\mu_{x}^{V} (y).
\]
\end{enumerate}
\end{definition}

The following result can be found in \cite[Theorem 1 p. 177]{ChungWalshBook}, for a standard Brownian motion on $\R^{d}$. The proof relies on the Markov property of the process, the semigroup property of the associated semigroup, and the definition of superharmonic functions. Thus it carries over to the setting of the current paper and therefore we do not give a proof.
We recall that $\left\{ g_{t} \right\}_{t}$ refers to the hypoelliptic Brownian motion, that is, the diffusion associated to $\mathcal{L}$.

\begin{theorem}\label{thm.ChungWalsh}
Suppose $D$ is a set such that $\overline{D}\subset \Omega$, and $u$ is an $\mathcal{L}$-superharmonic function defined in $\Omega$. Then
\[
\left\{ u\left( g_{\tau_{D} \wedge t} \right) \right\}_{t\geqslant 0}
\]
is a supermartingale under $\Prob^{x}$ for any $x\in D$ for which $u(x)<\infty$.
\end{theorem}

\begin{theorem}\label{thm.inequality}
Suppose $\Omega$ is an open bounded set, and $u$ is an $\mathcal{L}$-superharmonic function defined on $\Omega$. Then
\begin{equation}\label{eqn.superhar.estimate}
\E^{x} \left[u\left( g_{\tau_{D}} \right) \right] \leqslant u(x)
\end{equation}
for every $x \in \Omega$.
\end{theorem}

\begin{proof}
Let $\{ \Omega_{n} \}_{n \geqslant 	1}$ be a family of open bounded sets such that $\overline{\Omega}_{n} \subset \Omega$ and $\cup_{n=1}^{\infty} \Omega_{n} = \Omega$, and let $\tau_{n} := \tau_{\Omega_{n}}$. By Theorem \ref{thm.ChungWalsh} with $D=\Omega_{n}$ it follows that $\left\{ u\left( g_{\tau_{n} \wedge t} \right) \right\}_{t\geqslant 0}$ is a supermartingale and hence for any $t>0$
\[
\E^{x} \left[u\left( g_{t \wedge \tau_{n}} \right) \right] \leqslant u(x).
\]
Note that $\{ t < \tau_{n} \}  \nearrow\{ t< \tau_{\Omega} \}$ as $n\rightarrow \infty$, and hence if we let $n\rightarrow \infty$ by Fatou's Lemma the previous estimate becomes
\begin{equation}\label{eqn.estimate}
\E^{x} \left[u\left( g_{t \wedge \tau_{\Omega}} \right) \right] \leqslant u(x).
\end{equation}
Note that $\tau_{\Omega} < \infty$ $\Prob^{x}$-a.s. for any $x\in \Omega$. Indeed, $\{ \tau_{\Omega} = \infty \} = \cap _{M=1}^{\infty} \{ \tau_{\Omega} >M \}$, and hence by \eqref{eqn.bound.heat.kernel} for any $x\in \Omega$
\begin{align*}
& \Prob^{x} \left( \tau_{\Omega} = \infty \right) \leqslant \Prob^{x} \left( \tau_{\Omega} > M \right)
\\
& = \int_{\Omega} p^{\Omega}_{M} (x,y) dy \leqslant A \vert \Omega \vert M^{-\frac{Q}{2}},
\end{align*}
and by letting $M\rightarrow \infty$ it follows that $\Prob^{x} \left( \tau_{\Omega} = \infty \right)=0$. Thus the proof is completed by letting $t\rightarrow \infty$ in \eqref{eqn.estimate}.
\end{proof}

We now need the following version of \cite[Theorem 2.12, p. 245]{KaratzasShreveBMBook}.

\begin{proposition}\label{prop.regularity}
Let $y\in \partial \Omega$ and assume that
\[
\lim_{x\rightarrow y} \E^{x} \left[ f \left(  g_{\tau_{\Omega}}  \right) \right] = f(y),
\]
for every bounded measurable function $f: \partial \Omega \rightarrow \R$ which is continuous at $y$. Then $y$ is a regular point in the sense of Definition \ref{d.regular.set.prob}.
\end{proposition}

\begin{proof}
The proof given in \cite[Theorem 2.12 p. 245]{KaratzasShreveBMBook} holds for a standard Brownian motion in $\R^{d}$ with $d\geqslant 2$, but it only uses the Markov property and the fact that a standard Brownian motion never returns to its starting point when $d\geqslant 2$. The  hypoelliptic Brownian motion $g_{\, \cdot}$ is a Markov process that never returns to its starting point. Indeed, one can write
\[
g_{t} = \left( B_{t}, A_{2}(t) , \ldots , A_{r}(t) \right),
\]
where $B_{t}$ is a $d_{1}$-dimensional standard Brownian motion and $A_{j}(t) \in \R^{d_{j}}$ is an iterated stochastic integral for $j=2, \ldots,r$. Thus, if $g_{t}$ were to return to its starting point so would $B_{t}$, and that is not possible since $d_{1} \geqslant 2$.  The  proof of Proposition \ref{prop.regularity} then follows as in  \cite[Theorem 2.12 p. 245]{KaratzasShreveBMBook}.
\end{proof}

\begin{definition}\label{def.barrier}
Let $y\in \partial \Omega$. An $\mathcal{L}$-barrier at $y$ in $\Omega$ is a superharmonic map $w: \Omega \rightarrow (-\infty, + \infty]$ such that
\begin{enumerate}
\item $w(x)>0$ for every $x\in \Omega$.

\item $\lim_{x\rightarrow y} w(x) =0$.
\end{enumerate}
\end{definition}
In \cite[Theorem 6.10.4]{BonfiglioliLanconelliUguzzoniBook} it is shown that a point $y\in \partial \Omega$ is regular in the sense of Definition \ref{d.regular.set.an} if and only if there exists an $\mathcal{L}$-barrier at $y$ in $\Omega$. More precisely, for every regular point $y\in \partial \Omega$ one can construct an $\mathcal{L}$-barrier $s_{y}^{\Omega}$ such that
\begin{enumerate}
\item $s_{y}^{\Omega}$ is $\mathcal{L}$-harmonic in $\Omega$.

\item $\inf_{\Omega \backslash U} s_{y}^{\Omega}>0$ for every neighborhood $U$ of $y$.
\end{enumerate}
In particular, for every $z\in \partial \Omega$ with $z\neq y$ we have that
\[
\liminf_{x\rightarrow z} s_{y}^{\Omega} (x) >0.
\]
We can now prove the main Theorem of this section.

\begin{theorem}\label{thm.equivalence.reg.pts}
Let $\Omega$ be an open bounded connected set and let $y\in \partial \Omega$ be fixed. Then $y$ is regular in the sense of Definition \ref{d.regular.set.prob} if and only if is regular in the sense of Definition \ref{d.regular.set.an}.
\end{theorem}
\begin{proof}
To simplify the notation, we say that a point $y\in \partial \Omega$ is P-regular (A-regular) if it is regular in the sense of Definition \ref{d.regular.set.prob} (Definition \ref{d.regular.set.an}).

Let $y\in \partial \Omega$ be an A-regular point and $w(x):=s_{y}^{\Omega}(x)$ be the $\mathcal{L}$-barrier defined above. By Proposition \ref{prop.regularity} it is enough to show that
\[
\lim_{x\rightarrow y} \E^{x} \left[ f \left(  g_{\tau_{\Omega}}  \right) \right] = f(y),
\]
for every bounded measurable function $f: \partial \Omega \rightarrow \R$ which is continuous at $y$. The following argument is a modification of \cite[Proposition 2.15 p. 248]{KaratzasShreveBMBook}. Set $M:= \sup_{z\in \partial \Omega} \vert f(z) \vert$, and for any $\varepsilon>0$ let $\delta >0$ be such that $\vert f(z) - f(y) \vert \leqslant \varepsilon$ for any $z\in \partial \Omega$ with $d(z,y) < \delta$, where $d$ is a homogeneous distance. We know that $w(x)>0$ for every $x\in \Omega$ and $\liminf_{x\rightarrow z} w(x) >0$ for any $z\in \partial \Omega$ with $z\neq y$. Thus, there exists a $k$ such that $kw(x) \geqslant 2M$ for any $x\in \overline{\Omega}$ with $d(x,y) \geqslant \delta$. Thus, for any $z\in\partial \Omega$ we have that $\vert f(z) - f(y) \vert \leqslant \varepsilon$ if $d(z,y) < \delta$, and $\vert f(z) - f(y) \vert \leqslant 2M \leqslant k w(z)$ if $d(z,y) \geqslant \delta$. Hence, for any $z\in\partial \Omega$
\[
\vert f(z) - f(y) \vert \leqslant \max( \varepsilon , kw(z) ).
\]
$\mathcal{L}$-barriers are superharmonic and thus by Theorem \ref{thm.inequality} it follows that
\begin{align*}
& \vert \E^{x} \left[ f \left( g_{\tau_{\Omega}} \right) \right] - f(y) \vert \leqslant \E^{x} \left[ \vert  f \left( g_{\tau_{\Omega}} \right) - f(y) \vert \right]
\\
& \leqslant \E^{x} \left[  \max \left( \varepsilon, k w \left( g_{\tau_{\Omega}} \right) \right) \right] = \max \left( \varepsilon, k \E^{x} \left[ w \left( g_{\tau_{\Omega}} \right) \right]  \right)
\\
& \leqslant \max \left( \varepsilon, k w(x) \right),
\end{align*}
and then
\[
\limsup_{x\rightarrow y}\vert \E^{x} \left[ f \left( g_{\tau_{\Omega}} \right) \right] - f(y) \vert \leqslant \max \left( \varepsilon, k \limsup_{x\rightarrow y} w(x) \right) = \max (\varepsilon, 0) =\varepsilon,
\]
for any $\varepsilon>0$ and for any bounded measurable function $f: \partial \Omega \rightarrow \R$ which is continuous at $y$.

We now need to prove that P-regularity implies A-regularity. Let $y\in \partial \Omega$ be a P-regular point. By \cite[Theorem 6.10.4]{BonfiglioliLanconelliUguzzoniBook} it is enough to construct an $\mathcal{L}$-barrier at $y$ in $\Omega$. Following \cite[Exercise 10 p. 188]{ChungWalshBook}, it is easy to prove that $w(x) := \E^{x} \left[\tau_{\Omega} \right]$ is the desired $\mathcal{L}$-barrier.
\end{proof}

\section{Applications}\label{sec.applications}

\subsection{Small deviations}

Let $\Omega \subset \mathbb{G}$ be an open bounded connected set such that $e\in \Omega$, and for every $\varepsilon>0$ let
\begin{equation}\label{e.OmegaEps}
\Omega_{\varepsilon} := \delta_{\varepsilon} \left( \Omega \right),
\end{equation}
where $\delta_{\varepsilon} : \mathbb{G}\longrightarrow \mathbb{G}$ is the group dilation. In this section we describe how the spectral results from Section \ref{sec.spectral.prop} can be applied to find the asymptotic of the exit time $\tau_{\Omega_{\varepsilon}}$ of $g_{t}$ from $\Omega_{\varepsilon}$ as $\varepsilon \rightarrow 0$.

First, let $p_{t}$ be the heat kernel given by \eqref{eqn.global.tran.prob}. Then, $p_{t}$ satisfies the following scaling property \cite[Theorem 3.1 (i)]{Folland1975a}, for any $\varepsilon>0$
\begin{equation}\label{eqn.global.heatkernel.scaling}
p_{\frac{t}{\varepsilon^{2}}} (x,y) = \varepsilon^{Q} p_{t} \left(\delta_{\varepsilon} (x), \delta_{\varepsilon} (y) \right),
\end{equation}
for any $x,y \in \mathbb{G}$, where $Q$ denotes the homogeneous dimension of $\mathbb{G}$.

\begin{remark}[Space-time scaling in homogeneous Carnot groups]\label{rmk.scaling.process}
Let $g_{t}$ be a hypoelliptic Brownian motion. Then for any $x\in \mathbb{G}$ and for any $\varepsilon>0$ we have that
\begin{equation}\label{eqn.scaling.processs}
g^{x}_{\frac{t}{\varepsilon^{2}}} \stackrel{(d)}{=} \delta_{\frac{1}{\varepsilon}} \left( g_{t}^{\delta_{\varepsilon} (x)} \right).
\end{equation}
Indeed, by \eqref{eqn.global.heatkernel.scaling} for any Borel set $A \subset \mathbb{G}$
\begin{align*}
& \Prob \left( g^{x}_{\frac{t}{\varepsilon^{2}}} \in A \right)= \Prob^{x} \left(  g_{\frac{t}{\varepsilon^{2}}} \in A\right)
\\
& = \int_{A} p_{\frac{t}{\varepsilon^{2}}} (x,y) dy = \varepsilon^{Q} \int_{A} p_{t} \left( \delta_{\varepsilon} (x), \delta_{\varepsilon}(y) \right)dy
\\
& = \int_{\delta_{\varepsilon} (A)} p_{t} \left( \delta_{\varepsilon} (x), z \right) dz = \Prob^{\delta_{\varepsilon} (x) } \left(  g_{t} \in \delta_{\varepsilon} (A) \right) = \Prob \left(    \delta_{\frac{1}{\varepsilon}} \left( g_{t}^{\delta_{\varepsilon} (x)} \right) \in A   \right).
\end{align*}
\end{remark}

\begin{lemma}\label{lemma.scaling.dirichleth.HK}
Let $\Omega$ be an open set and $p_{t}^{\Omega}$ be the Dirichlet heat kernel. Then for any $\varepsilon >0$, and any $x,y \in \Omega$
\begin{equation}\label{eqn.dirichlet.heatkernel.scaling}
p^{\Omega}_{\frac{t}{\varepsilon^{2}}} (x,y) = \varepsilon^{Q} p^{\Omega_{\varepsilon}}_{t} \left(\delta_{\varepsilon} (x), \delta_{\varepsilon} (y) \right),
\end{equation}
where $\Omega_{\varepsilon}$ is defined in \eqref{e.OmegaEps}.
\end{lemma}

\begin{proof}
First, note that
\begin{align}\label{eqn.exit.time.scaling}
\mathbbm{1}_{  \left\{ \tau^{x}_{\Omega} > \frac{t}{\varepsilon^{2}} \right\} }   \stackrel{(d)}{=} \mathbbm{1}_{ \left\{ \tau^{\delta_{\varepsilon}(x)}_{\Omega_{\varepsilon}} > t \right\} }.
\end{align}
Indeed, by Remark \ref{rmk.scaling.process}
\begin{align*}
&\Prob \left( \tau^{x}_{\Omega} > \frac{t}{\varepsilon^{2}}  \right) = \Prob^{x} \left( g_{s} \in \Omega \text{ for all }  0\leqslant s\leqslant \frac{t}{\varepsilon^{2}} \right)
\\
& = \Prob\left( X^{x}_{ \frac{s}{\varepsilon^{2}}} \in \Omega  \text{ for all}  0\leqslant s\leqslant t \right)
\\
& = \Prob\left( \delta_{\frac{1}{\varepsilon}} \left( g_{s}^{\delta_{\varepsilon} (x)} \right) \in \Omega \; \text{ for all} \; 0\leqslant s\leqslant t \right)
\\
& =  \Prob^{\delta_{\varepsilon} (x)} \left( g_{s} \in \Omega_{\varepsilon} \; \text{ for all} \; 0\leqslant s\leqslant t  \right)
 = \Prob \left( \tau^{\delta_{\varepsilon}(x)}_{\Omega_{\varepsilon}} > t  \right).
\end{align*}
Thus, for any $f \in L^{2} (\Omega, dx)$ we have that
\begin{align*}
& \int_{\Omega} f(y) p^{\Omega}_{\frac{t}{\varepsilon^{2}}}(x,y)dy =  \E^{x} \left[ f\left( g_{\frac{t}{\varepsilon^{2}}} \right), \tau_{\Omega} > \frac{t}{\varepsilon^{2}}   \right]= \E \left[ f\left( X^{x}_{\frac{t}{\varepsilon^{2}}} \right), \tau^{x}_{\Omega} > \frac{t}{\varepsilon^{2}}   \right]
\\
& = \E \left[ f\left( \delta_{\frac{1}{\varepsilon}}  ( g_{t}^{\delta_{\varepsilon} (x)})  \right),  \tau^{\delta_{\varepsilon}(x)}_{\Omega_{\varepsilon}} > t \right] = \E^{\delta_{\varepsilon} (x) } \left[   f\left( \delta_{\frac{1}{\varepsilon}}  ( g_{t} ) \right),   \tau_{\Omega_{\varepsilon}} > t         \right]
\\
& = \int_{\Omega_{\varepsilon}} f \left( \delta_{\frac{1}{\varepsilon}}   (z)\right)p_{t}^{\Omega_{\varepsilon}} (\delta_{\varepsilon} (x),z )dz = \int_{\Omega} f(y) \varepsilon^{Q} p^{\Omega_{\varepsilon}}_{t} \left(\delta_{\varepsilon} (x), \delta_{\varepsilon} (y) \right) dy,
\end{align*}
which completes the proof.
\end{proof}

We conclude with an application to small deviations.

\begin{theorem}\label{thm.application}
Let $\mathbb{G}$ be a homogeneous Carnot group with the sub-Laplacian $\mathcal{L}$, and $\Omega$ be a bounded open connected set containing the identity $e$, and set  $ \Omega_{\varepsilon} := \delta_{\varepsilon} \left( \Omega \right)$. Let $g_{t}$ be a hypoelliptic Brownian motion such that $g_{0} = e$ a.s. Then
\begin{align*}
\lim_{\varepsilon \rightarrow 0} e^{\frac{\lambda_{1}}{\varepsilon^2} t} \Prob^{e} \left( \tau_{\Omega_{\varepsilon}}  >t \right)= c \phi (e),
\end{align*}
where $\lambda_{1}$ is the spectral gap of $-\mathcal{L}_{\Omega}$ given by Theorem \ref{thm.spectral.results}, and $\phi$ is the corresponding positive eigenfunction given by Theorem \ref{thm.simple.evalue}, and $c = \int_{\Omega} \phi(y) dy$.
\end{theorem}

\begin{corollary}\label{cor.endgoal}
Under the same assumption of Theorem \ref{thm.application} we have that
\begin{equation}\label{eqn.endgoal}
\lim_{\varepsilon \rightarrow 0} - \varepsilon^{2} \log \Prob^{e} \left( \tau_{\Omega_{\varepsilon}} >t \right) = \lambda_{1} t,
\end{equation}
for every $t>0$.
\end{corollary}

\begin{example}\label{example}
Let $\vert \cdot \vert$ be a homogeneous norm on $\mathbb{G}$. Then
\begin{align*}
& \lim_{\varepsilon \rightarrow 0} e^{\frac{\lambda_{1}}{\varepsilon^2} t} \Prob^{e} \left( \max_{0\leqslant s \leqslant t} \vert g_{s} \vert < \varepsilon \right)= c \phi (e),
\\
& \lim_{\varepsilon \rightarrow 0} - \varepsilon^{2} \log \Prob^{e} \left( \max_{0\leqslant s \leqslant t} \vert g_{s} \vert < \varepsilon \right) = \lambda_{1} t,
\end{align*}
where $\lambda_{1} >0$ is the spectral gap of $-\mathcal{L}_{B}$ and $B:= \{ x\in \mathbb{G}, \; \vert x \vert < 1  \}$.
\end{example}

\begin{remark}[Spectral gap estimates]\label{rmk.spectralgap.estimates}
If $\mathbb{G}=\Hei$ is the Heisenberg group, it is shown in \cite[Theorem 3.4]{CarfagniniGordina2022} that
\[
\lim_{\varepsilon \rightarrow 0} - \varepsilon^{2} \log \Prob^{e} \left( \max_{0\leqslant s \leqslant t} \vert g_{s} \vert < \varepsilon \right) = c^{2}t,
\]
for some finite constant $c>0$.  Moreover, Example \ref{example} and \cite[Theorem 5.6]{CarfagniniGordina2022} provide an explicit estimate for the first Dirichlet eigenvalue $\lambda_{1} =c^{2}$ for the sub-Laplacian on $\Hei$ in the Kor\'anyi ball. More precisely,
\[
\lambda_{1}^{(2)}  \leqslant \lambda_{1} \leqslant c\left(\lambda_{1}^{(1)}, \lambda_{1}^{(2)} \right),
\]
where
\begin{align*}
& c\left(\lambda_{1}^{(1)}, \lambda_{1}^{(2)} \right) := f \left( x^{\ast} \right) = \inf_{x\in (0,1)} f(x),
\\
& f(x) = \frac{ \lambda_{1}^{(2)}  }{ \sqrt{1-x}} + \frac{ \lambda_{1}^{(1)} \sqrt{1-x} }{4x},
\\
& x^{\ast} = \frac{\sqrt{ \left(\lambda_{1}^{(1)}\right)^{2} +32 \lambda_{1}^{(1)} \lambda_{1}^{(2)}} -3 \lambda_{1}^{(1)} }{2 \left( 4 \lambda_{1}^{(2)} - \lambda_{1}^{(1)} \right)},
\end{align*}
and $\lambda_{1}^{(n)}$ are the lowest Dirichlet eigenvalues of $-\frac{1}{2} \Delta_{\R^{n}}$ in the unit ball in $\R^{n}$.
\end{remark}

\begin{proof}[Proof of Theorem \ref{thm.application}]
By \eqref{eqn.exit.time.scaling} we have that
\begin{align*}
\Prob^{\delta_{\varepsilon} (x) } \left( \tau_{\Omega_{\varepsilon}} >t \right)=\Prob^{x} \left( \tau_{\Omega} > \frac{t}{\varepsilon^{2}} \right)
\end{align*}
for any $x\in \Omega$. Thus, by Corollary \ref{cor.important} we have that
\begin{align*}
\Prob^{e} \left( \tau_{\Omega_{\varepsilon}} >t \right) = \sum_{n=1}^{\infty} e^{-\lambda_{n} \frac{t}{\varepsilon^{2}}}  c_{n} \phi_{n} (e),
\end{align*}
where $\{ \phi_{n} \}_{n=1}^{\infty}$ and $\{ \lambda_{n} \}_{n=1}^{\infty}$ are defined as in Notation \ref{notation.basis} with $c_{n} = \int_{\Omega} \phi_{n} (y) dy$. By Theorem \ref{thm.simple.evalue} and Theorem \ref{thm.spectral.results} there exists a $\phi >0$ such that $\ker ( e^{-\lambda_{1} t} - P^{\Omega}_{t}  ) = \ker ( -\lambda_{1} - \mathcal{L}_{\Omega}  )= \operatorname{Span} \{ \phi \}$. Thus,
\begin{align*}
 e^{\lambda_{1} \frac{t}{\varepsilon^{2}}}\Prob^{e} \left( \tau_{\Omega_{\varepsilon}} >t \right)  = c \phi(e) + \sum_{n=2}^{\infty} e^{-( \lambda_{n} - \lambda_{1}) \frac{t}{\varepsilon^{2}}} c_{n} \phi_{n} (e),
\end{align*}
where $\lambda_{n} \geqslant \lambda_{2}$ for all $n\geqslant 3$, and $\lambda_{2} > \lambda_{1}$ since $\dim  \ker ( -\lambda_{1} - \mathcal{L}_{\Omega}  )=1$. Thus, the result follows by letting $\varepsilon$ go to zero.
\end{proof}

\begin{proof}[Proof of Corollary \ref{cor.endgoal}]
By Theorem \ref{thm.application} we know that
\begin{align*}
\lim_{\varepsilon \rightarrow 0} e^{\frac{\lambda_{1}}{\varepsilon^2} t} \Prob^{e} \left( \tau_{\Omega_{\varepsilon}}  >t \right)= c \phi (e),
\end{align*}
with $c \phi (e)>0$. Then
\begin{align}
& \log ( c\phi(e) ) = \lim_{\varepsilon \rightarrow 0} \log \left(  e^{\frac{\lambda_{1}}{\varepsilon^2} t} \Prob^{e} \left( \tau_{\Omega_{\varepsilon}}  >t \right) \right) \notag
\\
& = \lim_{\varepsilon \rightarrow 0} \left( \log \Prob^{e} \left( \tau_{\Omega_{\varepsilon}}  >t \right)  +   \frac{\lambda_{1}}{\varepsilon^2} t  \right)   = \lim_{\varepsilon \rightarrow 0} \left(   \frac{ \varepsilon^{2} \log \Prob^{e} \left( \tau_{\Omega_{\varepsilon}}  >t \right) + \lambda_{1}t  }{\varepsilon^{2} }   \right), \notag
\end{align}
which is finite if and only if
\begin{align*}
 \lim_{\varepsilon \rightarrow 0} \left(  \varepsilon^{2} \log \Prob^{e} \left( \tau_{\Omega_{\varepsilon}}  >t \right) + \lambda_{1}t \right) =0,
\end{align*}
that is,
\begin{align*}
 \lim_{\varepsilon \rightarrow 0}  - \varepsilon^{2} \log \Prob^{e} \left( \tau_{\Omega_{\varepsilon}}  >t \right) = \lambda_{1}t .
\end{align*}
\end{proof}

\subsection{Large time behavior of the heat content}

In this Section we use the spectral analysis from Section \ref{sec.spectral.prop} to describe the large time behavior of the heat content. Let $\Omega$ be a bounded open connected regular set. We consider the Dirichlet problem for the heat equation on $\Omega$

\begin{align}\label{eqn.par.BVP}
& \left( \partial_{t} - \mathcal{L} \right) u(x,t) =0, &  (t,x)\in (0,\infty) \times \Omega, \notag
\\
& u(t,x) =0, & (t,x)\in (0,\infty) \times \partial \Omega,
\\
& u(0,x)=1, & x\in \Omega. \notag
\end{align}

\begin{definition}\label{def.heat.content}
Let $u$ be the solution to the boundary value problem \eqref{eqn.par.BVP}. The heat content associated with $\Omega$ is given by
\begin{align*}
Q_{\Omega}(t) := \int_{\Omega} u (t,x) dx,
\end{align*}
for $t>0$.
\end{definition}
If $\Omega$ is regular, it is easy to see that $\Prob^{x} \left( \tau_{\Omega} >t \right)$ is the solution to \eqref{eqn.par.BVP}, and hence we can write
\[
Q_{\Omega} (t) = \int_{\Omega} \int_{\Omega} p^{\Omega}_{t} (x,y) dydx.
\]
By Corollary \ref{cor.important} we have that
\begin{equation}\label{eqn.heat.content}
Q_{\Omega} (t) = \int_{\Omega}  \sum_{n=1}^{\infty} e^{-\lambda_{n} t} c_{n} \phi_{n} (x)dx,
\end{equation}
where $c_{n} = \int_{\Omega} \phi_{n} (y)dy$. Note that the series $\sum_{n=1}^{\infty} e^{-\lambda_{n} t} c_{n} \phi_{n} (x)$ converges uniformly on $[\varepsilon, \infty) \times \Omega$ for any $\varepsilon>0$. Indeed, by Proposition \ref{prop.efunctions.properties} we have that, for any $x\in \Omega$ and $t\geqslant\varepsilon$
\begin{align*}
&\vert e^{-\lambda_{n} t} c_{n} \phi_{n} (x) \vert \leqslant e^{-\lambda_{n} t} c_{n} \Vert \phi_{n} \Vert_{L^{\infty} (\Omega, dx)}
\\
& \leqslant \vert \Omega \vert e^{-\lambda_{n} t} \Vert \phi_{n} \Vert_{L^{\infty} (\Omega, dx)}^{2} \leqslant\vert \Omega\vert d(\Omega) e^{-\lambda_{n} \varepsilon} \lambda_{n}^{Q}
\end{align*}
for any $n\in \mathbb{N}$. Thus $\sum_{n=1}^{\infty} e^{-\lambda_{n} t} c_{n} \phi_{n} (x)$ converges uniformly on the set $[\varepsilon, \infty) \times \Omega$ for any $\varepsilon>0$ by Weierstra{\ss}' M-test since the series $\sum_{n=1}^{\infty} e^{-\lambda_{n} \varepsilon} \lambda_{n}^{Q}$ is convergent.

Thus by \eqref{eqn.heat.content} it follows that
\begin{align*}
& Q_{\Omega} (t) = \int_{\Omega}  \sum_{n=1}^{\infty} e^{-\lambda_{n} t} c_{n} \phi_{n} (x) dx
\\
& = \sum_{n=1}^{\infty} e^{-\lambda_{n} t} c_{n}  \int_{\Omega}  \phi_{n} (x)dx = \sum_{n=1}^{\infty} e^{-\lambda_{n} t} c_{n}^{2},
\end{align*}
for any $t>0$. We can then deduce the following large time asymptotics for the heat content
\[
\lim_{t\rightarrow \infty}  e^{\lambda_{1} t }Q_{\Omega}(t) = c_{1}^{2}.
\]

\begin{acknowledgement}
The authors are grateful to Bruce Driver for posing the question. We thank Sasha Teplyaev for  helpful discussions (especially about Dirichlet forms) during the preparation of this work. Zhen-Qing Chen clarified a number of our questions about Dirichlet forms, and directed us to references \cite{ChungWalshBook, KaratzasShreveBMBook}. The first author thanks Luca Capogna and  Ermanno Lanconelli for suggesting useful references. Peter Friz pointed out the connection to similar problems and techniques used in the theory of rough paths. Finally we acknowledge useful discussions with Fabrice Baudoin, Lorenzo Dello Schiavo, Patrick Fitzsimmons, and Tommaso Rossi.
\end{acknowledgement}


\begin{thebibliography}{10}

\bibitem{ArendtGraboschGreinerMoustakasNagelNeubranderSchlotterbeck1986}
W.~Arendt, A.~Grabosch, G.~Greiner, U.~Groh, H.~P. Lotz, U.~Moustakas,
  R.~Nagel, F.~Neubrander, and U.~Schlotterbeck, \emph{One-parameter semigroups
  of positive operators}, Lecture Notes in Mathematics, vol. 1184,
  Springer-Verlag, Berlin, 1986. \MR{839450}

\bibitem{BatkaiKramarRhandiBook2017}
Andr\'{a}s B\'{a}tkai, Marjeta Kramar~Fijav\v{z}, and Abdelaziz Rhandi,
  \emph{Positive operator semigroups}, Operator Theory: Advances and
  Applications, vol. 257, Birkh\"{a}user/Springer, Cham, 2017, From finite to
  infinite dimensions, With a foreword by Rainer Nagel and Ulf Schlotterbeck.
  \MR{3616245}

\bibitem{BonfiglioliLanconelliUguzzoniBook}
A.~Bonfiglioli, E.~Lanconelli, and F.~Uguzzoni, \emph{Stratified {L}ie groups
  and potential theory for their sub-{L}aplacians}, Springer Monographs in
  Mathematics, Springer, Berlin, 2007. \MR{2363343}

\bibitem{ByczkowskiLectureNotes2013}
Tomasz Byczkowski, \emph{Potential theory of subordinated {B}rownian motions},
  Tech. report, {I}nstitute of {M}athematics of {P}olish {A}cademy of
  {S}ciences, {J}oint {D}octoral {S}tudies {P}roject ({SSDNM}), 2013.

\bibitem{CarfagniniGordina2022}
Marco Carfagnini and Maria Gordina, \emph{Small deviations and {C}hung's law of
  iterated logarithm for a hypoelliptic {B}rownian motion on the {H}eisenberg
  group}, Trans. Amer. Math. Soc. Ser. B \textbf{9} (2022), 322--342.
  \MR{4410042}

\bibitem{ChungWalshBook}
Kai~Lai Chung and John~B. Walsh, \emph{Markov processes, {B}rownian motion, and
  time symmetry}, second ed., Grundlehren der mathematischen Wissenschaften
  [Fundamental Principles of Mathematical Sciences], vol. 249, Springer, New
  York, 2005. \MR{2152573}

\bibitem{ColindeVerdiereHillairetTrelat2021}
Yves Colin~de Verdi\`ere, Luc Hillairet, and Emmanuel Tr\'{e}lat,
  \emph{Small-time asymptotics of hypoelliptic heat kernels near the diagonal,
  nilpotentization and related results}, Ann. H. Lebesgue \textbf{4} (2021),
  897--971. \MR{4315774}

\bibitem{CorwinGreenleafBook}
Lawrence~J. Corwin and Frederick~P. Greenleaf, \emph{Representations of
  nilpotent {L}ie groups and their applications. {P}art {I}}, Cambridge Studies
  in Advanced Mathematics, vol.~18, Cambridge University Press, Cambridge,
  1990, Basic theory and examples. \MR{1070979 (92b:22007)}

\bibitem{Doob1954}
J.~L. Doob, \emph{Semimartingales and subharmonic functions}, Trans. Amer.
  Math. Soc. \textbf{77} (1954), 86--121. \MR{64347}

\bibitem{DriverGrossSaloff-Coste2009a}
Bruce~K. Driver, Leonard Gross, and Laurent Saloff-Coste, \emph{Holomorphic
  functions and subelliptic heat kernels over {L}ie groups}, J. Eur. Math. Soc.
  (JEMS) \textbf{11} (2009), no.~5, 941--978. \MR{2538496 (2010h:32052)}

\bibitem{Dynkin1959a}
E.~B. Dynkin, \emph{The natural topology and excessive functions connected with
  a {M}arkov process}, Dokl. Akad. Nauk SSSR \textbf{127} (1959), 17--19.
  \MR{0107303}

\bibitem{EvansPDEBook2nd}
Lawrence~C. Evans, \emph{Partial differential equations}, second ed., Graduate
  Studies in Mathematics, vol.~19, American Mathematical Society, Providence,
  RI, 2010. \MR{2597943}

\bibitem{Folland1975a}
G.~B. Folland, \emph{Subelliptic estimates and function spaces on nilpotent
  {L}ie groups}, Ark. Mat. \textbf{13} (1975), no.~2, 161--207. \MR{0494315}

\bibitem{FranceschiPrandiRizzi2020}
Valentina Franceschi, Dario Prandi, and Luca Rizzi, \emph{On the essential
  self-adjointness of singular sub-{L}aplacians}, Potential Anal. \textbf{53}
  (2020), no.~1, 89--112. \MR{4117982}

\bibitem{FrizVictoir2008}
Peter Friz and Nicolas Victoir, \emph{On uniformly subelliptic operators and
  stochastic area}, Probab. Theory Related Fields \textbf{142} (2008), no.~3-4,
  475--523. \MR{2438699}

\bibitem{FrizVictoirBook2010}
Peter~K. Friz and Nicolas~B. Victoir, \emph{Multidimensional stochastic
  processes as rough paths}, Cambridge Studies in Advanced Mathematics, vol.
  120, Cambridge University Press, Cambridge, 2010, Theory and applications.
  \MR{2604669}

\bibitem{FukushimaOshimaTakedaBook2011}
Masatoshi Fukushima, Yoichi Oshima, and Masayoshi Takeda, \emph{Dirichlet forms
  and symmetric {M}arkov processes}, extended ed., De Gruyter Studies in
  Mathematics, vol.~19, Walter de Gruyter \& Co., Berlin, 2011. \MR{2778606}

\bibitem{Gaveau1977a}
Bernard Gaveau, \emph{Principe de moindre action, propagation de la chaleur et
  estim\'ees sous elliptiques sur certains groupes nilpotents}, Acta Math.
  \textbf{139} (1977), no.~1-2, 95--153. \MR{0461589 (57 \#1574)}

\bibitem{GordinaLaetsch2016a}
Maria Gordina and Thomas Laetsch, \emph{Sub-{L}aplacians on {S}ub-{R}iemannian
  {M}anifolds}, Potential Anal. \textbf{44} (2016), no.~4, 811--837.
  \MR{3490551}

\bibitem{GordinaLaetsch2017}
\bysame, \emph{A convergence to {B}rownian motion on sub-{R}iemannian
  manifolds}, Trans. Amer. Math. Soc. \textbf{369} (2017), no.~9, 6263--6278.
  \MR{3660220}

\bibitem{Grigor'yan2010a}
Alexander Grigor'yan, \emph{Heat kernels on metric measure spaces with regular
  volume growth}, Handbook of geometric analysis, {N}o. 2, Adv. Lect. Math.
  (ALM), vol.~13, Int. Press, Somerville, MA, 2010, pp.~1--60. \MR{2743439}

\bibitem{GrigoryanHuLau2014}
Alexander Grigor'yan, Jiaxin Hu, and Ka-Sing Lau, \emph{Heat kernels on metric
  measure spaces}, Geometry and analysis of fractals, Springer Proc. Math.
  Stat., vol.~88, Springer, Heidelberg, 2014, pp.~147--207. \MR{3276002}

\bibitem{HansenHueber1987}
W.~Hansen and H.~Hueber, \emph{The {D}irichlet problem for sub-{L}aplacians on
  nilpotent {L}ie groups---geometric criteria for regularity}, Math. Ann.
  \textbf{276} (1987), no.~4, 537--547. \MR{879533 (88g:31017)}

\bibitem{Hormander1967a}
Lars H{\"o}rmander, \emph{Hypoelliptic second order differential equations},
  Acta Math. \textbf{119} (1967), 147--171. \MR{0222474 (36 \#5526)}

\bibitem{JerisonSanchez-Calle1987}
David Jerison and Antonio S\'anchez-Calle, \emph{Subelliptic, second order
  differential operators}, Complex analysis, {III} ({C}ollege {P}ark, {M}d.,
  1985--86), Lecture Notes in Math., vol. 1277, Springer, Berlin, 1987,
  pp.~46--77. \MR{922334}

\bibitem{JerisonSanchez-Calle1986}
David~S. Jerison and Antonio S\'{a}nchez-Calle, \emph{Estimates for the heat
  kernel for a sum of squares of vector fields}, Indiana Univ. Math. J.
  \textbf{35} (1986), no.~4, 835--854. \MR{865430}

\bibitem{Kac1951}
M.~Kac, \emph{On some connections between probability theory and differential
  and integral equations}, Proceedings of the Second Berkeley Symposium on
  Mathematical Statistics and Probability, 1950 (Berkeley and Los Angeles),
  University of California Press, 1951, pp.~189--215. \MR{MR0045333 (13,568b)}

\bibitem{KallianpurSundar2014}
Gopinath Kallianpur and P.~Sundar, \emph{Stochastic analysis and diffusion
  processes}, Oxford Graduate Texts in Mathematics, vol.~24, Oxford University
  Press, Oxford, 2014. \MR{3156223}

\bibitem{KaratzasShreveBMBook}
Ioannis Karatzas and Steven~E. Shreve, \emph{Brownian motion and stochastic
  calculus}, second ed., Graduate Texts in Mathematics, vol. 113,
  Springer-Verlag, New York, 1991. \MR{1121940 (92h:60127)}

\bibitem{KreinRutman1948}
M.~G. Kre\u{\i}n and M.~A. Rutman, \emph{Linear operators leaving invariant a
  cone in a {B}anach space}, Uspehi Matem. Nauk (N. S.) \textbf{3} (1948),
  no.~1(23), 3--95. \MR{0027128}

\bibitem{MaRocknerBook}
Zhi~Ming Ma and Michael R{\"o}ckner, \emph{Introduction to the theory of
  (nonsymmetric) {D}irichlet forms}, Universitext, Springer-Verlag, Berlin,
  1992. \MR{1214375 (94d:60119)}

\bibitem{NagelSteinWainger1985}
Alexander Nagel, Elias~M. Stein, and Stephen Wainger, \emph{Balls and metrics
  defined by vector fields. {I}. {B}asic properties}, Acta Math. \textbf{155}
  (1985), no.~1-2, 103--147. \MR{793239 (86k:46049)}

\bibitem{PrandiRizziSeri2019}
Dario Prandi, Luca Rizzi, and Marcello Seri, \emph{A sub-{R}iemannian
  {S}antal\'{o} formula with applications to isoperimetric inequalities and
  first {D}irichlet eigenvalue of hypoelliptic operators}, J. Differential
  Geom. \textbf{111} (2019), no.~2, 339--379. \MR{3909911}

\bibitem{ReedSimonIV}
Michael Reed and Barry Simon, \emph{Methods of modern mathematical physics.
  {IV}. {A}nalysis of operators}, Academic Press [Harcourt Brace Jovanovich,
  Publishers], New York-London, 1978. \MR{0493421}

\bibitem{RizziRossi2021}
Luca Rizzi and Tommaso Rossi, \emph{Heat content asymptotics for
  sub-{R}iemannian manifolds}, J. Math. Pures Appl. (9) \textbf{148} (2021),
  267--307. \MR{4223354}

\bibitem{RuzhanskySuragan2017}
Michael Ruzhansky and Durvudkhan Suragan, \emph{Layer potentials, {K}ac's
  problem, and refined {H}ardy inequality on homogeneous {C}arnot groups}, Adv.
  Math. \textbf{308} (2017), 483--528. \MR{3600064}

\bibitem{SchillingPartzschBook2012}
Ren\'{e}~L. Schilling and Lothar Partzsch, \emph{Brownian motion}, De Gruyter,
  Berlin, 2012, An introduction to stochastic processes, With a chapter on
  simulation by Bj\"{o}rn B\"{o}ttcher. \MR{2962168}

\bibitem{SoggeBook2017}
Christopher~D. Sogge, \emph{Fourier integrals in classical analysis}, second
  ed., Cambridge Tracts in Mathematics, vol. 210, Cambridge University Press,
  Cambridge, 2017. \MR{3645429}

\bibitem{Strichartz1986a}
Robert~S. Strichartz, \emph{Sub-{R}iemannian geometry}, J. Differential Geom.
  \textbf{24} (1986), no.~2, 221--263. \MR{862049 (88b:53055)}

\bibitem{Sturm1996a}
K.~T. Sturm, \emph{Analysis on local {D}irichlet spaces. {III}. {T}he parabolic
  {H}arnack inequality}, J. Math. Pures Appl. (9) \textbf{75} (1996), no.~3,
  273--297. \MR{1387522}

\bibitem{Sturm1995a}
Karl-Theodor Sturm, \emph{Analysis on local {D}irichlet spaces. {II}. {U}pper
  {G}aussian estimates for the fundamental solutions of parabolic equations},
  Osaka J. Math. \textbf{32} (1995), no.~2, 275--312. \MR{1355744}

\bibitem{Sturm1995b}
\bysame, \emph{On the geometry defined by {D}irichlet forms}, Seminar on
  {S}tochastic {A}nalysis, {R}andom {F}ields and {A}pplications ({A}scona,
  1993), Progr. Probab., vol.~36, Birkh\"{a}user, Basel, 1995, pp.~231--242.
  \MR{1360279}

\bibitem{TysonWangJ2018}
Jeremy Tyson and Jing Wang, \emph{Heat content and horizontal mean curvature on
  the {H}eisenberg group}, Comm. Partial Differential Equations \textbf{43}
  (2018), no.~3, 467--505. \MR{3804205}

\bibitem{VaropoulosBook1992}
N.~Th. Varopoulos, L.~Saloff-Coste, and T.~Coulhon, \emph{Analysis and geometry
  on groups}, Cambridge University Press, Cambridge, 1992. \MR{95f:43008}

\bibitem{ZelditchBook2017}
Steve Zelditch, \emph{Eigenfunctions of the {L}aplacian on a {R}iemannian
  manifold}, CBMS Regional Conference Series in Mathematics, vol. 125,
  Published for the Conference Board of the Mathematical Sciences, Washington,
  DC; by the American Mathematical Society, Providence, RI, 2017. \MR{3729409}

\end{thebibliography}
\end{document}